\documentclass{article}[12pt]
\usepackage{color}
\usepackage{tikz}
\usetikzlibrary {intersections}
\usetikzlibrary {intersections,through}
\usepackage[title]{appendix}
\usepackage{mathrsfs}
\usepackage{mathtools}
\usepackage{graphicx}
\usepackage{epstopdf}
\usepackage{stmaryrd}
 \usepackage{subfigure}
\usepackage{float}
\usepackage{amsmath}
    \usepackage{bm}
\usepackage{amsfonts,amssymb}
  \usepackage{dsfont}
\usepackage{amsfonts}
\usepackage{mathrsfs}
  \usepackage{pifont}
\numberwithin{equation}{section}
 \usepackage{amsthm}
 \newtheorem{theorem}{Theorem}[section]
\newtheorem{lemma}[theorem]{Lemma}

\newtheorem{corollary}[theorem]{Corollary}
\newtheorem{remark}[theorem]{Remark}

 \usepackage{amsmath}
  \usepackage[hidelinks]{hyperref}
 \usepackage{multirow}

\begin{document}
\title{A Pressure-Robust Nonconforming Immersed Finite Element Method for Stokes Interface Problems}
\author{
Haifeng Ji\footnotemark[1]$~^,$\footnotemark[2] \qquad
Feng Wang\footnotemark[2]$~^,$\footnotemark[3]
}
\footnotetext[1]{Jiangsu Key Laboratory of Quantum Computing Science and Devices, School of Science, Nanjing University of Posts and Telecommunications, Nanjing, Jiangsu 210023, China  (hfji1988@foxmail.com)}
\footnotetext[2]{Key Laboratory of NSLSCS, Ministry of Education, School of Mathematical Sciences, Nanjing Normal University, Nanjing, Jiangsu 210023, China (fwang@njnu.edu.cn)}
\footnotetext[3]{Corresponding author}

\date{}
\maketitle

\begin{abstract}
It is well established that an appropriate modification of test functions  may lead to pressure-robust mixed methods for Stokes problems.  However,  for  immersed finite element approximations of Stokes interface problems on unfitted meshes, it remains unclear  whether the velocity error is independent of the pressure, since the velocity and  the pressure are coupled in one of the interface conditions.  In this paper, we provide a  positive answer through a novel decomposition of the discontinuous pressure into a continuous component and a velocity-dependent discontinuous component.
We demonstrate that the immersed Crouzeix--Raviart/$P_0$ element method achieves pressure robustness via an $H(\operatorname{div})$-conforming reconstruction of the test functions on the right-hand side.   The stability and optimal error estimates of the proposed  method are  established with constants independent of the interface position relative to the mesh. Numerical experiments are presented to validate the theoretical findings. 
\end{abstract}



\section{Introduction} 
Two-phase incompressible flows arise in many scientific and engineering applications, such as multiphase fluid dynamics, biomedical flows, and material processing. When inertial effects are negligible, the Stokes interface problem serves as a suitable model while retaining the essential challenges in designing effective numerical methods.
Let $\Omega \subset \mathbb{R}^2$ be a bounded convex domain with polygonal boundary $\partial \Omega$, and let $\Gamma$ be a simple  closed $C^2$  curve immersed in $\Omega$, separating it into two subdomains $\Omega^+$ and $\Omega^-$.
The Stokes interface problem reads: find a velocity field $\boldsymbol{u}$ and a pressure field $p$ such that
\begin{subequations}\label{originalpb0}
\begin{align}
-\nabla\cdot \boldsymbol{\sigma}(\mu,\boldsymbol{u},p) &=\boldsymbol{f}  \qquad\mbox{in } \Omega^+\cup\Omega^-,\label{originalpb1}\\
\nabla\cdot \boldsymbol{u}&=0\qquad\mbox{in } \Omega,\label{originalpb2}\\
[\boldsymbol{\sigma}(\mu,\boldsymbol{u},p)\boldsymbol{n}]_\Gamma&=\boldsymbol{0}\qquad\mbox{on } \Gamma,\label{jp_cond1}\\
[\boldsymbol{u}]_\Gamma&=\boldsymbol{0}\qquad\mbox{on } \Gamma,\label{jp_cond2}\\
\boldsymbol{u}&=\boldsymbol{0} \qquad\mbox{on } \partial\Omega,\label{originalpb5}
\end{align}
\end{subequations}
where $\boldsymbol{\sigma}(\mu,\boldsymbol{u},p)=2\mu \boldsymbol{\epsilon}(\boldsymbol{u})-p\mathbb{I}$ denotes  the total stress tensor, $\boldsymbol{\epsilon}(\boldsymbol{u})=\frac{1}{2}(\nabla \boldsymbol{u}+(\nabla \boldsymbol{u})^\top)$ represents the strain tensor,   $\mathbb{I}$ is the identity matrix, $\boldsymbol{n}$ is the unit normal to  $\Gamma$ pointing from $\Omega^-$ to $\Omega^+$, $[\boldsymbol{v}]_\Gamma$ stands for  the jump of $\boldsymbol{v}$ across $\Gamma$, i.e., $[\boldsymbol{v}]_\Gamma=\boldsymbol{v}^+|_\Gamma-\boldsymbol{v}^-|_\Gamma$ with $\boldsymbol{v}^\pm=\boldsymbol{v}|_{\Omega^\pm}$,  $\boldsymbol{f} \in L^2(\Omega)^2$ is a given body force, and $\mu$ is a piecewise constant viscosity with $\mu|_{\Omega^\pm} = \mu^\pm > 0$.
If the traces of $\nabla\cdot\boldsymbol u^\pm$ on $\Gamma$ are well defined, then \eqref{originalpb2} provides an additional  relationship
\begin{equation}\label{jp_cond3}
[\nabla \cdot \boldsymbol{u}]_\Gamma=0\quad\mbox{ on }\Gamma.
\end{equation}
We assume, without loss of generality, that $\partial \Omega^- = \Gamma$, and that there exists a positive constant $\varepsilon_0$ such that $\mathrm{dist}(\Gamma, \partial \Omega) \geq \varepsilon_0$.
For clarity of presentation, we first restrict our attention to the homogeneous interface condition \eqref{jp_cond1}. The extension to the case with surface tension is deferred to Section~\ref{sec_exten}.

For interface problems, unfitted mesh finite element methods, in which the mesh does not align with the interface, have attracted considerable attention due to their flexibility in handling complex or moving interfaces.
To ensure stability and accuracy of the finite element method, special treatment is required for interface elements that are cut by the interface.
In general, two main strategies exist for developing unfitted mesh finite element methods with optimal convergence: one enriches the standard  finite element space by introducing additional degrees of freedom on interface elements to resolve discontinuities (see, e.g., \cite{hansbo2014cut, cattaneo2015stabilized, kirchhart2016analysis, guzman2018inf, caceres2020new, badia2018mixed, burman2021unfitted}), while the other modifies the standard  finite element space to strongly incorporate interface conditions without altering the original degrees of freedom. The immersed finite element (IFE) method is a typical representative of the latter approach. 
The first IFE method was developed in \cite{li1998immersed} for one-dimensional elliptic interface problems, and since then, the method has been extensively extended and analyzed in numerous studies \cite{li2004immersed, he2012convergence, taolin2015siam, GuzmanJSC2017, Guojcp2020, 2021ji_IFE, ji3Dnonconforming}.
A key advantage of the IFE method over other unfitted mesh finite element methods is that its IFE space is isomorphic to the traditional finite element space on the same mesh,  independent of the interface location with respect to the mesh.
In this work, this property plays a crucial role in our development of a pressure-robust IFE method for the Stokes interface problem, as it eliminates the need for a pressure penalty term.

Developing and analyzing IFE methods for the Stokes interface problem is considerably more challenging than for elliptic interface problems, due to the coupling of velocity and pressure in the interface conditions.
The first IFE method for the Stokes interface problem was developed in \cite{adjerid2015immersed}, where velocity and pressure were treated as a single vector-valued function to enforce the coupled interface conditions. Subsequent developments include nonconforming IFE methods on triangular and rectangular meshes \cite{jones2021class}, the Taylor--Hood IFE method \cite{chen2021p2}, and the MINI IFE method \cite{ji2025mini}. The nonconforming IFE method based on the standard Crouzeix--Raviart element was analyzed in \cite{2021ji_IFE_stoke}. 
However, these methods are not pressure robust: the velocity error depends on the pressure, and modifying the body force by a gradient field, which changes only the pressure, may substantially affect the discrete velocity \cite{john2017divergence}. 
For standard Stokes equations without an interface, there are various ways to design pressure-robust schemes, for example, by constructing divergence-free elements \cite{scott1985norm,zhang2005new}. 
Recently, several studies \cite{liu2023cutfem,burman2024cut,frachon2024divergence} have been conducted on divergence-free cut finite element methods for solving  Stokes equations with smooth boundaries, where additional degrees of freedom are required and the boundary conditions are weakly enforced. 
When solving the interface problem within the IFE framework, it appears challenging to construct IFEs on interface elements that simultaneously satisfy the coupled interface conditions, the divergence-free constraint, and discrete inf-sup stability. 
In \cite{zhu2024divergence}, the authors proposed a divergence-free Petrov-Galerkin immersed finite element method (PG-IFEM), in which the velocity space is enriched by the lowest-order Raviart--Thomas element space. Nevertheless, establishing the inf-sup condition for PG-IFEM remains challenging. 
Guo et al.~\cite{guo2023solving} established the inf-sup stability of
a related PG-IFEM for elliptic interface problems under suitable assumptions
on the coefficient contrast.
An alternative approach to achieving pressure robustness is to replace the test function on the right-hand side of the discretization with its divergence-conforming reconstruction \cite{LINKE2014782}.

In this work, we  develop a pressure-robust  IFE method for the Stokes interface problem, based on the nonconforming IFE  introduced in \cite{2021ji_IFE_stoke}.
The nonconforming IFE is a modification of the classical Crouzeix--Raviart/$P_0$ element   \cite{crouzeix1973conforming}, 
which is a classical lowest-order stable finite element pair for the Stokes problem.
The interface conditions are strongly enforced on the approximate interface  without altering the standard integral-type degrees of freedom.  As  a result, the  discrete velocity is discretely divergence-free. However, the method in \cite{2021ji_IFE_stoke} is not  pressure-robust, as illustrated by the numerical experiments in Section~\ref{sec_num}.

It seems difficult to obtain a pressure-robust scheme using the approach developed in  \cite{LINKE2014782}, since  the  discrete velocity and pressure are coupled in the IFE space due to the interface condition \eqref{jp_cond1}. By revisiting the interface condition, we write the pressure $p$ as a continuous part $\hat{p}$ and a discontinuous part  $R(\boldsymbol{u})$, which depends on the velocity $\boldsymbol{u}$ and  satisfies $  [R(\boldsymbol{u})\boldsymbol{n}]_{\Gamma} = [\boldsymbol{\sigma}(\mu,\boldsymbol{u},0)\boldsymbol{n}]_\Gamma $.  Thus, the interface condition \eqref{jp_cond1} can be rewritten as
$  [\boldsymbol{\sigma}(\mu,\boldsymbol{u},R(\boldsymbol{u}))\boldsymbol{n}]_\Gamma =\boldsymbol{0}$, which depends only on the velocity.
Fortunately, the discrete pressure can also be decomposed into a standard piecewise-constant part and a discrete velocity-dependent part.
With this decomposition, we prove that the proposed IFE method is inf-sup stable without requiring a pressure-penalty term commonly used in existing IFE formulations. This enables us to define the space of discretely divergence-free functions and thereby reformulate the discrete method as an elliptic problem involving only the velocity.
Another key ingredient in establishing pressure robustness is the commutative diagram property of certain interpolation operators. 
Using the discrete interface conditions, we show that this property also holds for the IFE interpolation operators.
 By carefully estimating the consistency error arising from the variational crime of using a lowest-order Raviart--Thomas velocity reconstruction, we derive optimal pressure-robust  error estimates for the proposed IFE method, with constants independent of the position of the interface relative to the mesh.

The proposed method differs from that in   \cite{2021ji_IFE_stoke}, since
the additional pressure stabilization term $J_h(\cdot,\cdot)$  in \cite{2021ji_IFE_stoke} is not needed, and we employ a divergence-conforming  reconstruction for the test function on the right-hand side.  Moreover, we develop a new technique—the decomposition of the discrete pressure into  the velocity-dependent component $R_h(\boldsymbol{u}_h)$ and the piecewise constant component $\hat{p}_h$. This decomposition provides a new perspective for understanding and establishing the pressure robustness of the method.
 
The outline of the paper is as follows.
In Section~\ref{sec_pre}, we introduce the necessary preliminaries. In Section~\ref{sec_FEM}, we formulate our pressure-robust IFE method. In Section~\ref{sec_pro},  several key properties of the IFE space are established. In Section~\ref{sec_analy}, we present the theoretical analysis of our IFE method.
 In Section~\ref{sec_exten}, we extend the IFE method to problems with surface tension.
In Section~\ref{sec_num}, we provide numerical examples to validate the theoretical findings. Finally, conclusions are drawn in Section~\ref{sec_con}. 

\section{Preliminaries}\label{sec_pre}
Let $k\geq 0$ be an integer and $1\leq p\leq \infty$.  We adopt the standard notation $W^k_p(D)$ for Sobolev spaces on a domain $D$, with the norm $\|\cdot\|_{W^k_p(D)}$ and the seminorm $|\cdot|_{W^k_p(D)}$. In particular, $W^k_2(D)$ is denoted by $H^{k}(D)$ with the norm  $\|\cdot\|_{H^{k}(D)}$ and the seminorm $|\cdot|_{H^{k}(D)}$. 
If $D\cap \Omega \not=\emptyset$, we define subdomains $D^\pm=D\cap \Omega^\pm$ and broken Sobolev spaces
\begin{equation*}
H^k(\cup D^\pm)=\{v\in L^2(D) : v|_{D^\pm} \in H^k(D^\pm)\},
\end{equation*}
equipped with the norm $\|\cdot\|_{H^k(\cup D^\pm)}$ and the  seminorm $|\cdot|_{H^k(\cup D^\pm)}$, satisfying
$$
\|\cdot\|^2_{H^k(\cup D^\pm)}=\|\cdot\|^2_{H^k(D^+)}+\|\cdot\|^2_{H^k(D^-)}, \quad|\cdot|^2_{H^k(\cup D^\pm)}=|\cdot|^2_{H^k(D^+)}+|\cdot|^2_{H^k(D^-)}.
$$

As usual, $H_0^1(D)=\{v\in H^1(D) : v=0 \mbox{ on }\partial D\}$ and $L_0^2(D)=\{q\in L^2(D) : \int_D q=0\}$.
Let $\boldsymbol{V}=H_0^1(\Omega)^2$  and $Q=L_0^2(\Omega)$. The weak formulation of (\ref{originalpb0}) reads: find $(\boldsymbol{u}, p)\in \boldsymbol{V}\times Q$  such that
\begin{equation}\label{weakform0}
A((\boldsymbol{u},p),(\boldsymbol{v},q)):=a(\boldsymbol{u}, \boldsymbol{v})+b(\boldsymbol{v}, p)-b(\boldsymbol{u}, q) = l(\boldsymbol{v})\quad \forall \,  (\boldsymbol{v},q) \in \boldsymbol{V}\times Q,
\end{equation}
where
$a(\boldsymbol{u}, \boldsymbol{v})=\int_\Omega2\mu\boldsymbol{\epsilon}(\boldsymbol{u}):\boldsymbol{\epsilon}(\boldsymbol{v})$, $b(\boldsymbol{v}, q)=-\int_\Omega q\nabla\cdot \boldsymbol{v}$ and $l(\boldsymbol{v})=\int_\Omega \boldsymbol{f}\cdot\boldsymbol{v}.$ 

The well-posedness of this weak formulation can be found in \cite{gross2011numerical}.
For the convergence analysis, we assume that the solution has a higher regularity in each
subdomain, i.e., $(\boldsymbol{u}, p)\in \widetilde{\boldsymbol{H}^2H^1}$ with
\begin{equation*}
\widetilde{\boldsymbol{H}^2H^1}=\{ (\boldsymbol{v}, q)\in H^2(\cup\Omega^\pm)^2\times H^1(\cup\Omega^\pm) : [\boldsymbol{\sigma}(\mu,\boldsymbol{v},q)\boldsymbol{n}]_\Gamma=0, [\boldsymbol{v}]_\Gamma=\boldsymbol{0}, [\nabla\cdot \boldsymbol{v}]_\Gamma=0 \}.
\end{equation*}

For any $(\boldsymbol{v},q)\in \widetilde{\boldsymbol{H}^2H^1}\cap (\boldsymbol{V}\times Q)$, we construct a new function $\tilde{q}=R(\boldsymbol{v})$, with $R$ being a linear operator, such that 
\begin{equation}\label{cons_qprime}
 (\boldsymbol{v},\tilde{q})\in \widetilde{\boldsymbol{H}^2H^1}\cap (\boldsymbol{V}\times Q)~\mbox{ and }~\|\tilde{q} \|_{H^1(\cup \Omega^\pm)}\leq C\|\boldsymbol{v}\|_{H^2(\cup\Omega^\pm)}.
\end{equation}
In fact, the function $\tilde{q}$ can be constructed by setting $\tilde{q}|_{\Omega^\pm}=w^\pm-c$, where  $w^\pm$ satisfy 
\begin{equation*}
w^+=0 ~~\mbox{ and }~~ \Delta w^-=0 \mbox{ in } \Omega^-,\quad w^-|_{\Gamma}=-[\boldsymbol{\sigma}(\mu,\boldsymbol{v},0)\boldsymbol{n}]_\Gamma\cdot\boldsymbol{n},
\end{equation*}
and the constant $c$ is chosen so that $\tilde{q}$ has average zero. Using the regularity results for Laplace’s equation (see \cite{Gilbargbook}), the inequality in \eqref{cons_qprime} follows.
Letting $\hat{q}=q-\tilde{q}$,  we obtain a decomposition
$$
(\boldsymbol{v},q)=(\boldsymbol{v},R(\boldsymbol{v}))+(\boldsymbol{0},\hat{q}).
$$
By construction, we have $[\hat{q}]_\Gamma=0$, and thus
\begin{equation}\label{hat_q_b}
R(\boldsymbol{v})\in L_0^2(\Omega),\qquad \hat{q}\in H^1(\Omega)\cap L_0^2(\Omega).
\end{equation}
Conversely, for any $r\in H^1(\Omega)\cap L_0^2(\Omega)$, it is easy to verify that 
$$(\boldsymbol{v},R(\boldsymbol{v}))+(\boldsymbol{0},r)\in \widetilde{\boldsymbol{H}^2H^1}\cap (\boldsymbol{V}\times Q).$$
We now define the following spaces:
$$
\boldsymbol{X}=\{ \boldsymbol{v}:  (\boldsymbol{v},R(\boldsymbol{v}))\in \widetilde{\boldsymbol{H}^2H^1}\cap (\boldsymbol{V}\times Q) \}, \quad Y=H^1(\Omega)\cap L_0^2(\Omega).
$$
Based on the above derivation, we obtain the following space decomposition:
\begin{equation}\label{decomp_con}
\widetilde{\boldsymbol{H}^2H^1}\cap (\boldsymbol{V}\times Q)=\left \{(\boldsymbol{v},R(\boldsymbol{v})) : \boldsymbol{v} \in \boldsymbol{X}\right\} \oplus (\{\boldsymbol{0}\}\times Y).
\end{equation}

If we choose any $(\boldsymbol{v}, q) \in \widetilde{\boldsymbol{H}^2H^1} \cap (\boldsymbol{V} \times Q)$ as test functions in the weak formulation \eqref{weakform0}, then based on the above decomposition, 
the exact solution $(\boldsymbol{u}, p)$ with $p=R(\boldsymbol{u})+\hat{p}$ and $(\boldsymbol{u}, \hat{p})\in  \boldsymbol{X}\times Y$ satisfies the following relations:
\begin{equation}\label{weakform1}
\begin{aligned}
A(\,(\boldsymbol{u},R(\boldsymbol{u})), (\boldsymbol{v},R(\boldsymbol{v})) \,)+ b(\boldsymbol{v}, \hat{p})&=l(\boldsymbol{v}) \quad &&\forall \,  \boldsymbol{v}\in \boldsymbol{X},\\
b(\boldsymbol{u},\hat{q})&=0\quad &&\forall\,   \hat{q}\in Y.
\end{aligned}
\end{equation}

\section{The Finite  Element  Method}\label{sec_FEM}
\subsection{Unfitted Meshes}
Let $ \{\mathcal{T}_h\}_{h>0}$ be a family of shape-regular  triangulations of the domain $\Omega$, generated independently of the interface $\Gamma$. The mesh size of $\mathcal{T}_h$ is defined by $h=\max_{T\in\mathcal{T}_h}h_T$ with $h_T=\mathrm{diam}(T)$ being the diameter of $T$.  
We denote the set of edges of $\mathcal{T}_h$ by $\mathcal{E}_h$. 
The sets of interior and boundary edges are denoted  by $\mathcal{E}^\circ_h$ and  $\mathcal{E}^\partial_h$, respectively.
We adopt the convention that elements $T$ and edges $e$ are open sets. Then, the sets of interface elements and interface edges are defined by $\mathcal{T}_h^\Gamma =\{T\in\mathcal{T}_h :  T\cap \Gamma\not = \emptyset\}$ and $\mathcal{E}_h^\Gamma=\{ e\in \mathcal{E}_h : e \cap \Gamma\not = \emptyset\}$, respectively. The sets of non-interface elements and non-interface edges are denoted by $\mathcal{T}_h^{non}=\mathcal{T}_h\backslash\mathcal{T}_h^\Gamma$ and $\mathcal{E}_h^{non}=\mathcal{E}_h\backslash \mathcal{E}_h^\Gamma$, respectively.  
For each edge $e \in \mathcal{E}_h$, we denote by $\mathcal{T}_h^e$ the set of all elements in $\mathcal{T}_h$ that share $e$ as an edge.

We assume that $h$ is sufficiently small such that the interface $\Gamma$ intersects the boundary of any element $T$ at most twice, and intersects the closure $\overline{e}$ of any edge $e$ at most once.
We approximate the interface $\Gamma$  by a piecewise linear interface $\Gamma_h$, which consists of all line segments connecting the intersection points of the exact interface with the boundaries of the interface elements.  The approximate interface $\Gamma_h$ divides $\Omega$ into two subdomains: $\Omega_h^+$ and $\Omega_h^-$, with $\partial \Omega_h^-=\Gamma_h$.    
On each interface element $T\in\mathcal{T}_h^\Gamma$,  the discrete interface $\Gamma_h$ divides $T$ into two sub-elements:
$T^+_h=T\cap \Omega^+_h$ and $T^-_h=T\cap \Omega^-_h.$ We denote 
$\Gamma_T=\Gamma\cap T$ and $\Gamma_{h,T}=\Gamma_h\cap T.$
Let $\textbf{n}_h(\boldsymbol{x})$  be the unit normal vector to $\Gamma_h$, oriented outward from $\Omega^-_h$ toward $\Omega^+_h$; see Figure~\ref{fig_inter} for an illustration.
The corresponding unit tangent vectors to $\Gamma_h$ and $\Gamma$ are denoted by $\boldsymbol{t}_h$ and $\boldsymbol{t}$, respectively, and are obtained by a $\pi/2$ clockwise rotation of $\boldsymbol{n}_h$ and $\boldsymbol{n}$.

\begin{figure}[htbp]
\centering
\begin{tikzpicture}[scale=1.6]
\draw [thick] (-1,-1)--(1,-1);
\draw [thick, name path=e2] (-1,-1)--(0.4,0.6);
\draw [thick, name path=e1] (0.4,0.6)--(1,-1);
\draw [thick](-1+1.4/2,-1+1.6/2)--(1-0.6/4,-1+1.6/4);
\node [below left] at (-1,-1) {$A_1$};
\node [below right ]at (1,-1) {$A_2$};
\node [above]at (0.4,0.6) {$A_3$};
\node [above left] at (1-0.6/4,-1+1.6/4) {$D$};
\node [below ] at (-1+1.4/2,-1+1.6/2) {$E$};
\draw [dashed, thick](-1+1.4/2-0.5,-1+1.6/2-0.2) to [out=25,in=180+20] (-1+1.4/2,-1+1.6/2)
to [out=20,in=150] (-1+1.4/2+0.4,-1+1.6/2)
to [out=-30,in=160] (-1+1.4/2+0.8,-1+1.6/2-0.5)
to [out=-20,in=180+30] (1-0.6/4,-1+1.6/4)
to [out=30,in=180+40] (1-0.6/4+0.5,-1+1.6/4+0.4);
\node [above] at (-1+1.4/2-0.5,-1+1.6/2-0.2) {$\Gamma$};
\node at (-1+1.4/2+0.6,-1+1.6/2+0.2) {$T_h^+$};
\node at (-1+1.4/2+0.2,-1+1.6/2-0.5) {$T_h^-$};
\draw [dotted, very thick] (0.85,-0.6)--(0.85+1.15,-0.6+-0.4);
\draw [ ->,thick] (0.85+1.15*0.4,-0.6-0.4*0.4)--(0.85+1.15*0.6,-0.6-0.4*0.6);
\draw [ ->,thick] (0.85+1.15*0.4,-0.6-0.4*0.4)--(0.85+1.15*0.4+0.4*0.2 ,-0.6-0.4*0.4 +1.15*0.2);
\node [above right] at (0.85+1.15*0.4+0.07,-0.6-0.4*0.4+0.1) {$\boldsymbol{n}_h$};
\node [below right] at (0.85+1.15*0.4,-0.6-0.4*0.4) {$\boldsymbol{t}_h$};
\end{tikzpicture}
\caption{An interface element\label{fig_inter}} 
\end{figure}
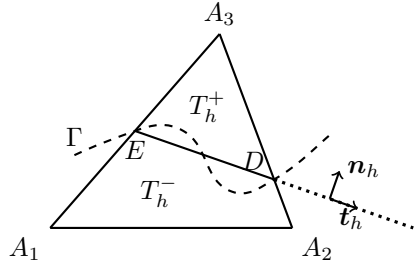

Denote by $\operatorname{dist}(\boldsymbol{x}, \Gamma)$ the Euclidean distance between a point $\boldsymbol{x}$ and the interface $\Gamma$, and define the tubular neighborhood of $\Gamma$ with thickness $\delta$ as
 $$U(\Gamma,\delta)=\{ \boldsymbol{x}\in\mathbb{R}^2: \mbox{dist}( \boldsymbol{x},\Gamma)< \delta\}.$$
 We define the  signed distance function  $d( \boldsymbol{x})$ by $d( \boldsymbol{x})|_{\Omega^\pm}=\pm \mbox{dist}( \boldsymbol{x},\Gamma)$.
 Under the assumption that $\Gamma$ is $C^2$-smooth, there exists a constant $\delta_0 > 0$ such that $d(\boldsymbol{x}) \in C^2(U(\Gamma, \delta_0))$ (see \cite{foote1984regularity}). 
 We further assume that $h$ is sufficiently small so that $T\subset U(\Gamma,\delta_0)$ for all interface elements $T\in\mathcal{T}_h^\Gamma$. 
We define the region enclosed between the exact and approximate interfaces $\Gamma$ and $\Gamma_h$ as
\begin{equation}\label{def_triangleT}
T^\triangle=(T^-\cap T_h^+)\cup(T^+\cap  T_h^-)\qquad \forall T\in\mathcal{T}_h^\Gamma.
\end{equation}
Since $\Gamma$ is $C^2$-smooth, there exists a constant $C$ depending only on the curvature of $\Gamma$ such that 
\begin{equation}\label{triT_rela}
T^\triangle\subset U(\Gamma, Ch^2)\qquad \forall T\in\mathcal{T}_h^\Gamma.
\end{equation}
The following lemma provides a $\delta$-strip estimate, which will be used to control the error in the region near the interface; see Lemma 2.1 in \cite{Li2010Optimal}.
\begin{lemma}\label{strip}
Let $\delta>0$ be a sufficiently small number.  Then it holds  for any $v\in H^1(\Omega)$ that 
\begin{equation*}
\|v\|_{L^2(U(\Gamma,\delta))}\leq C\sqrt{\delta}\,  \| v\|_{H^1(\Omega)}.
\end{equation*}
\end{lemma}

\subsection{The Crouzeix--Raviart Immersed Finite  Element}
We recall the construction of the Crouzeix--Raviart immersed finite element (CR-IFE), introduced in \cite{jones2021class,2021ji_IFE_stoke}, which is a modification of the classical Crouzeix--Raviart element \cite{crouzeix1973conforming}.
Let $P_k(D)$ be the set of all polynomials of degree less than or equal to $k$ on a domain $D\subset\mathbb{R}^2$. 
Define  $D^\pm=D\cap \Omega^\pm$ and $P_k(\cup D^\pm)=\{v : v|_{D^\pm} \in P_k(D^\pm)\}$. Let $\boldsymbol{P}_k(D)=P_k(D)^2$ and $\boldsymbol{P}_k(\cup D^\pm)=P_k(\cup D^\pm)^2$. 
Given an element $T\in\mathcal{T}_h$, the standard  $CR$-$P_0$ shape function space is $\boldsymbol{P}_1(T)\times P_0(T)$.

On each interface element $T\in\mathcal{T}_h^\Gamma$, we need to modify $\boldsymbol{P}_1(T)\times P_0(T)$ since it cannot  capture the behavior of the exact solution well due to the interface jump conditions. 
According to (\ref{jp_cond1}), (\ref{jp_cond2}) and (\ref{jp_cond3}), the discrete interface jump conditions for all $ (\boldsymbol{v}, q) \in  \boldsymbol{P}_1(\cup T_h^\pm)\times P_0(\cup T_h^\pm)$ are introduced as follows:
\begin{subequations}\label{dis_jp0}
\begin{align}
& [\boldsymbol{\sigma}(\mu_h,\boldsymbol{v},q)\boldsymbol{n}_h]_{\Gamma_{h,T}}=\boldsymbol{0},~~ \label{dis_jp1}\\
&[\boldsymbol{v}]_{\Gamma_{h,T}}=\boldsymbol{0} \mbox{ (or, equivalently, $[\boldsymbol{v}]_{\Gamma_{h,T}}(\boldsymbol{x}_T^*)=\boldsymbol{0}$, $[(\nabla \boldsymbol{v})\boldsymbol{t}_{h}]_{\Gamma_{h,T}}=\boldsymbol{0}$)}, \label{dis_jp2}\\
&[\nabla\cdot \boldsymbol{v}]_{\Gamma_{h,T}}=0,\label{dis_jp3}
\end{align}
\end{subequations}
where $\mu_h|_{\Omega_h^\pm}=\mu^\pm$, $\boldsymbol{x}_T^*$ is a point on $\Gamma_{h,T}$, and $[ \cdot ]_{\Gamma_{h,T}}$ stands for the jump  across $\Gamma_{h,T}$.
The modified $CR$-$P_0$ shape function space is then defined as
\begin{equation*}
\widetilde{\boldsymbol{P}_1P_0}(T)=\{   (\boldsymbol{v}, q)\in  \boldsymbol{P}_1(\cup T_h^\pm)\times P_0(\cup T_h^\pm) :  (\boldsymbol{v}, q) \mbox{  satisfies  (\ref{dis_jp0})} \},
\end{equation*}
where we emphasize that the velocity and pressure components are coupled due to the discrete interface jump condition (\ref{dis_jp1}).

Let $e_{j,T}$, $j\in \{1,2,3\}$, denote the edges of the element $T$, and write $\boldsymbol{v}=(v_1,v_2)^\top$. For the spaces $\widetilde{\boldsymbol{P}_1P_0}(T)$ and $\boldsymbol{P}_1(T)\times P_0(T)$,  the degrees of freedom, denoted by ${\rm DoF}_{k,T}$, $k= 1,2,\cdots,7$, are defined as follows:
\begin{equation*}
\begin{aligned}
{\rm DoF}_{j+3(i-1),T}(\boldsymbol{v},q)&=\frac{1}{|e_{j,T}|}\int_{e_{j,T}}v_i \quad  j=1,2,3,~~ i=1,2,\\
{\rm DoF}_{7,T}(\boldsymbol{v},q)&=\frac{1}{|T|}\int_T q.
\end{aligned}
\end{equation*}
The global CR-IFE space is defined by
\begin{equation*}
\begin{aligned}
\widetilde{\boldsymbol{V}Q}_h^{\rm{IFE}}=\left \{ \frac{}{} (\boldsymbol{v},q) :  \right. &(\boldsymbol{v}, q)|_T\in \boldsymbol{P}_1(T)\times P_0(T) ~~\forall\, T\in\mathcal{T}_h^{non},  \\
 & (\boldsymbol{v}, q)|_T\in \widetilde{\boldsymbol{P}_1P_0}(T)~~\forall\, T\in\mathcal{T}_h^\Gamma,~~\int_{e}[\boldsymbol{v}]_e=\boldsymbol{0} ~~\forall\, e\in\mathcal{E}^\circ_h,\\
 &\left. \int_e\boldsymbol{v}=\boldsymbol{0} ~~\forall\, e\in\mathcal{E}_h^\partial, ~\int_\Omega q=0 \right \},
\end{aligned}
\end{equation*}
where $[\boldsymbol{v}]_e$ denotes the jump of $\boldsymbol{v}$ across the edge $e$.

\subsection{IFE Basis Functions and the Space Decomposition}
The standard CR finite element basis functions on an element $T$ are defined by
$$
(\boldsymbol{\phi}_{k,T}, \varphi_{k,T})\in \boldsymbol{P}_1(T)\times P_0(T),\quad 
{\rm DoF}_{l,T}(\boldsymbol{\phi}_{k,T}, \varphi_{k,T})=\delta_{kl}, \quad \forall\, k, l\in \{1,2,\cdots,7\},
$$
where $\delta_{kl} $ is the Kronecker symbol.
On each interface element $T\in\mathcal{T}_h^\Gamma$, the IFE basis functions are defined analogously by
$$
(\boldsymbol{\phi}_{k,T}^{\rm IFE}, \varphi_{k,T}^{\rm IFE})  \in \widetilde{\boldsymbol{P}_1P_0}(T), \quad  {\rm DoF}_{l,T}(\boldsymbol{\phi}_{k,T}^{\rm IFE}, \varphi_{k,T}^{\rm IFE})=\delta_{kl},  \quad \forall\, k, l\in \{1,2,\cdots,7\}.
$$

Define two functions as follows:
\begin{equation}\label{def_zw}
\begin{aligned}
z_T(\boldsymbol{x}) &=
\begin{cases}
-1 & \text{if } \boldsymbol{x} \in T_h^+, \\
\;\;0 & \text{if } \boldsymbol{x} \in T_h^-,
\end{cases}
\qquad
w_T(\boldsymbol{x}) =
\begin{cases}
\mathrm{dist}(\boldsymbol{x}, \Gamma_{h,T}^{\mathrm{ext}}) & \text{if } \boldsymbol{x} \in T_h^+, \\
\;\;0 & \text{if } \boldsymbol{x} \in T_h^-,
\end{cases}
\end{aligned}
\end{equation}
where $\Gamma_{h,T}^{\mathrm{ext}}$ is the extension of $\Gamma_{h,T}$ to the straight line that contains it.
Define the interpolation operators \( \pi^{\rm CR}_{h,T} : W(T) \rightarrow P_1(T) \) and \( \pi^{0}_{h,T} : L^2(T) \rightarrow P_0(T) \) by
\begin{equation*}
\int_{e_{j,T}} \pi^{\rm CR}_{h,T} v = \int_{e_{j,T}} v \quad \forall\, j \in \{1,2,3\}, 
\qquad 
\int_T \pi^{0}_{h,T} v = \int_T v,
\end{equation*}
where
$
W(T) = \{ v \in L^2(T) :  \int_{e_{j,T}} v \text{ is well defined for all } j=1, 2, 3 \}.
$
Then, we have the following result regarding the unisolvence of the IFE basis functions (see Lemma 4.6 and (4.37) in \cite{2021ji_IFE_stoke}).
\begin{lemma}\label{lem_IFEbasis}
For each \( T \in \mathcal{T}_h^\Gamma \), every function in \( \widetilde{\boldsymbol{P}_1P_0}(T) \) is uniquely determined by the degrees of freedom \( \mathrm{DoF}_{l,T} \), $l=1, 2, \dots, 7$, and the corresponding IFE basis functions can be explicitly constructed as follows:
\begin{equation*}
\begin{aligned}
\boldsymbol{\phi}_{k,T}^{\rm IFE} &=
\begin{cases}
\boldsymbol{\phi}_{k,T} + \displaystyle\frac{\boldsymbol{t}_{h}^\top \boldsymbol{\sigma}(\mu^-/\mu^+ - 1, \boldsymbol{\phi}_{k,T}, 0)\boldsymbol{n}_h}{1 + (\mu^-/\mu^+ - 1)\nabla \pi^{\rm CR}_{h,T} w_T \cdot \boldsymbol{n}_h} (w_T - \pi^{\rm CR}_{h,T} w_T)\boldsymbol{t}_{h}, & k = 1,2, \dots, 6, \\
\boldsymbol{0}, & k = 7,
\end{cases} \\[1.2ex]
\varphi_{k,T}^{\rm IFE} &=
\begin{cases}
\boldsymbol{n}_h^\top \boldsymbol{\sigma}(\mu^- - \mu^+, \boldsymbol{\phi}_{k,T}, 0)\boldsymbol{n}_h (z_T - \pi^{0}_{h,T} z_T), & k = 1, 2, \dots, 6, \\
\varphi_{k,T}, & k = 7.
\end{cases}
\end{aligned}
\end{equation*}
\end{lemma}

\begin{remark}
From Lemma~\ref{lem_IFEbasis}, we observe that the IFE basis function \( (\boldsymbol{\phi}_{k,T}^{\rm IFE}, \varphi_{k,T}^{\rm IFE}) \) differs from the standard CR finite element basis function \( (\boldsymbol{\phi}_{k,T}, \varphi_{k,T}) \) by correction terms arising from the enforcement of the discrete interface jump conditions. In particular, for the IFE basis functions associated with the degrees of freedom \( \mathrm{DoF}_{l,T} \), \( l = 1, 2, \dots, 6 \), both the velocity and pressure components involve such corrections. In contrast, no correction term is required for the basis function corresponding to the last degree of freedom.
This structure arises from the fact that the original coupled discrete jump condition \eqref{dis_jp1} is actually equivalent to the following  \cite{zhu2024divergence}:
\begin{subequations}\label{jp_equal}
\begin{align}
[(2\mu \boldsymbol{\epsilon}(\boldsymbol{v}) - p\mathbb{I})\boldsymbol{n}_h \cdot \boldsymbol{n}_h]_{\Gamma_{h,T}} &= 0,\label{jp_equal_1}\\
[(2\mu \boldsymbol{\epsilon}(\boldsymbol{v}) - p\mathbb{I})\boldsymbol{n}_h \cdot \boldsymbol{t}_{h}]_{\Gamma_{h,T}} 
= [2\mu \boldsymbol{\epsilon}(\boldsymbol{v})\boldsymbol{n}_h \cdot \boldsymbol{t}_{h}]_{\Gamma_{h,T}} &= 0, \label{jp_equal_2}
\end{align}
\end{subequations}
where the coupling between the velocity and pressure appears only in \eqref{jp_equal_1}.
As a consequence, the velocity component (with $12$ unknown parameters) can be uniquely determined using six interface constraints, namely  \eqref{jp_equal_2}, \eqref{dis_jp2}, and \eqref{dis_jp3}, together with  six degrees of freedom \( \mathrm{DoF}_{k,T} \) for all \( k \in \{1, 2, \cdots, 6\} \), without involving the pressure. The piecewise constant pressure is then obtained using \eqref{jp_equal_1} and the final degree of freedom \( \mathrm{DoF}_{7,T} \).
\end{remark}

According to the explicit expressions of IFE basis functions given in Lemma~\ref{lem_IFEbasis}, we define the linear operator $R_{h,T}$ for each interface element $T\in\mathcal{T}_h^\Gamma$ by
\begin{equation*}
 R_{h,T}(\boldsymbol{v})=\sum_{k=1}^{6}c_k\varphi_{k,T}^{\rm IFE}~ \mbox{ for all }~ \boldsymbol{v}=\sum_{k=1}^{6}c_k\boldsymbol{\phi}_{k,T}^{\rm IFE}.
\end{equation*}
We then define the linear operator $R_h$ by
\begin{equation} \label{def_Rh}
\left(R_h(\boldsymbol{v})\right)|_T= \left\{
\begin{aligned}
&R_{h,T}(\boldsymbol{v}|_T) \quad &&\mbox{ if }  T\in\mathcal{T}_h^\Gamma,\\
&\boldsymbol{0} \quad &&\mbox{ if } T\in\mathcal{T}_h^{non},
\end{aligned}
\qquad \mbox{ for all } (\boldsymbol{v},q)\in \widetilde{\boldsymbol{V}Q}_{h}^{\rm{IFE}}.
\right.
\end{equation}
It is easy to verify that 
\begin{equation}\label{pro_Rh}
R_{h,T}(\boldsymbol{v})\in L^2_0(T)\quad\mbox{ and }\quad R_h(\boldsymbol{v})\in L^2_0(\Omega).
\end{equation}
We now define the following spaces:
$$
\boldsymbol{X}_h^{\rm IFE}=\{ \boldsymbol{v}:  (\boldsymbol{v},R_h(\boldsymbol{v}))\in \widetilde{\boldsymbol{V}Q}_{h}^{\rm{IFE}} \}, \quad
Y_h=\{q \in L_0^2(\Omega) : q|_T\in P_0(T) ~\forall\, T\in \mathcal{T}_h\}.
$$
Similar to the continuous case (\ref{decomp_con}), we obtain the following discrete space decomposition:
\begin{equation}\label{decomp_dis}
\widetilde{\boldsymbol{V}Q}_{h}^{\rm{IFE}}=\left \{(\boldsymbol{v},R_h(\boldsymbol{v})) : \boldsymbol{v} \in \boldsymbol{X}_h^{\rm IFE}\right\} \oplus (\{\boldsymbol{0}\}\times Y_h).
\end{equation}
Let $\boldsymbol{X}_h$ be the standard CR finite element space defined as
\begin{equation*}
\boldsymbol{X}_h =\{ \boldsymbol{v} :  \boldsymbol{v}|_T\in \boldsymbol{P}_1(T)~~\forall\, T\in\mathcal{T}_h,~~\int_{e}[\boldsymbol{v}]_e=\boldsymbol{0} ~~\forall\, e\in\mathcal{E}^\circ_h, ~~\int_{e}\boldsymbol{v}=\boldsymbol{0} ~~\forall\, e\in\mathcal{E}^\partial_h  \},
\end{equation*}
Clearly,  $\boldsymbol{X}_h^{\rm IFE}$ is a modification of $\boldsymbol{X}_h$
, and it coincides with $\boldsymbol{X}_h$ when $\mu^+=\mu^-$.

\subsection{The Pressure-Robust Scheme}
For each edge $e\in \mathcal{E}_h$, we prescribe a unit normal vector $\boldsymbol{n}_e$. 
The orientations of interior edges are arbitrary but fixed, whereas the orientations of boundary edges are chosen to point outward from the domain $\Omega$.
For an edge $e\in \mathcal{E}_h^\circ$, there exist two elements $T_1$ and $T_2$ that share the edge. We assume that $\boldsymbol{n}_e$ is oriented from $T_1$ to $T_2$.
The jump and the average of a function $v$ across the edge $e$ are then defined as 
$[v]_e=v|_{T_1}-v|_{T_2}$ and  $\{v\}_e=\frac{1}{2}(v|_{T_1}+v|_{T_2})$, respectively. For boundary edges $e\in \mathcal{E}_h^\partial$, we define $[v]_e=\{v\}_e=v|_e$. For simplicity of notation, we omit the subscript $e$ when no confusion arises.

Following the idea proposed in \cite{LINKE2014782}, which modifies the right-hand side to enhance pressure robustness, we define the lowest-order Raviart--Thomas (RT) finite element space \cite{10.1007/BFb0064470} as 
\begin{equation*}
\boldsymbol{RT}_0(\mathcal{T}_h) = \left\{ \boldsymbol{v}_h \in H(\text{div}; \Omega) :  \boldsymbol{v}_h|_T= \boldsymbol{a}_T+b_T\boldsymbol{x}~ \mbox{ for all }  T \in \mathcal{T}_h, \boldsymbol{a}_T\in \mathbb{R}^2, b_T\in \mathbb{R} \right\}
\end{equation*}
and define the corresponding Raviart--Thomas interpolation operator $\boldsymbol{\pi}_h^{\rm RT}: \boldsymbol{V}+ \boldsymbol{X}_h^{\rm IFE}\rightarrow \boldsymbol{RT}_0(\mathcal{T}_h)$ by
\begin{equation*}
(\boldsymbol{\pi}_h^{\rm RT}\boldsymbol{v})|_e\cdot \boldsymbol{n}_e=\frac{1}{|e|}\int_e\{\boldsymbol{v}\}\cdot \boldsymbol{n}_e\qquad \forall\, e\in\mathcal{E}_h.
\end{equation*}
We now introduce the discrete bilinear and linear forms as follows:
\begin{equation}\label{def_Ah}
\begin{aligned}
&A_h(\,(\boldsymbol{u}_h, p_h), (\boldsymbol{v}_h, q_h) \,)=a_h(\boldsymbol{u}_h,\boldsymbol{v}_h)+b_h(\boldsymbol{v}_h,p_h)-b_h(\boldsymbol{u}_h,q_h),\\
&a_h(\boldsymbol{u}_h,\boldsymbol{v}_h)=\sum_{T\in\mathcal{T}_h}\int_T2\mu_h \boldsymbol{\epsilon}(\boldsymbol{u}_h):\boldsymbol{\epsilon}(\boldsymbol{v}_h)+\sum_{e\in\mathcal{E}_h}\frac{\mu_{\rm max}}{h_e}\int_e[\boldsymbol{u}_h]\cdot[\boldsymbol{v}_h]\\
&\quad-\sum_{e\in\mathcal{E}_h^\Gamma}\int_e \left(\{ 2\mu_h\boldsymbol{\epsilon}(\boldsymbol{u}_h)\boldsymbol{n}_e\}\cdot[\boldsymbol{v}_h]+\theta\{ 2\mu_h\boldsymbol{\epsilon}(\boldsymbol{v}_h)\boldsymbol{n}_e\}\cdot[\boldsymbol{u}_h]\right) +\sum_{e\in\mathcal{E}_h^\Gamma}\frac{\eta\mu_{\rm max}}{h_e}\int_e[\boldsymbol{u}_h]\cdot[\boldsymbol{v}_h],\\
&b_h(\boldsymbol{v}_h,q_h)=-\sum_{T\in\mathcal{T}_h}\int_T q_h\nabla\cdot\boldsymbol{v}_h +\sum_{e\in\mathcal{E}_h^\Gamma}\int_e\{q_h\}[\boldsymbol{v}_h]\cdot\boldsymbol{n}_e,\\
&l_h(\boldsymbol{v}_h)=\int_\Omega \boldsymbol{f}\cdot \boldsymbol{\pi}_h^{\rm RT}\boldsymbol{v}_h,
\end{aligned}
\end{equation}
where
$
(\boldsymbol{u}_h, p_h), (\boldsymbol{v}_h, q_h) \in  \widetilde{\boldsymbol{H}^2H^1}+\widetilde{\boldsymbol{V}Q}_{h}^{\rm{IFE}}$, $\mu_{\rm max}=\max\{\mu^+,\mu^-\}$,  $\theta=\pm 1$, $\eta\ge 0$, and $h_e$ denotes the diameter of the edge $e$.
When the parameter $\theta=1$, the bilinear form $a_h(\cdot,\cdot)$ is symmetric and the penalty parameter $\eta$ should be sufficiently large to ensure the coercivity.  When $\theta=-1$, the bilinear form $a_h(\cdot,\cdot)$ is non-symmetric, and we can choose an arbitrary $\eta\geq0$ to ensure the coercivity.  
For an explanation of the terms included on edges, we refer the reader to  \cite{2021ji_IFE_stoke,ji2023analysis}.

Based on  \eqref{weakform1}, we formulate the pressure-robust CR-IFE method for the Stokes interface problem as follows: find $(\boldsymbol{u}_h, p_h)$, with $p_h=R_h(\boldsymbol{u}_h)+\hat{p}_h$ and $(\boldsymbol{u}_h, \hat{p}_h)\in  \boldsymbol{X}_h^{\rm IFE}\times Y_h$, such that 
\begin{equation}\label{IFE_method}
\begin{aligned}
A_h(\,(\boldsymbol{u}_h,R_h(\boldsymbol{u}_h)), (\boldsymbol{v}_h,R_h(\boldsymbol{v}_h))\, )+ b_h(\boldsymbol{v}_h, \hat{p}_h)&=l_h(\boldsymbol{v}_h) \quad &&\forall \,  \boldsymbol{v}_h\in \boldsymbol{X}_h^{\rm IFE},\\
b_h(\boldsymbol{u}_h,\hat{q}_h)&=0\quad &&\forall\,   \hat{q}_h\in Y_h.
\end{aligned}
\end{equation}

\begin{remark}\label{remark_bh}
For all $\hat{q}_h\in Y_h$,  the average $\{\hat{q}_h\}$ is constant on each edge $e$. Hence,
$$\int_e\{\hat{q}_h\}[\boldsymbol{v}_h]\cdot\boldsymbol{n}_e=0,\qquad \forall e\, \in\mathcal{E}_h^\Gamma,~\hat{q}_h\in Y_h,~ \boldsymbol{v}_h\in \boldsymbol{X}_h^{\rm IFE}.$$
As a result, the bilinear form $b_h(\cdot, \cdot)$ in the proposed IFE method \eqref{IFE_method} simplifies to
\begin{equation*}
b_h(\boldsymbol{v}_h, \hat{p}_h)=-\sum_{T\in\mathcal{T}_h}\int_T \hat{p}_h\nabla\cdot\boldsymbol{v}_h,\qquad b_h(\boldsymbol{u}_h, \hat{q}_h)=-\sum_{T\in\mathcal{T}_h}\int_T \hat{q}_h\nabla\cdot\boldsymbol{u}_h.
\end{equation*}
Nevertheless, the edge integrals in $b_h(\cdot, \cdot)$ appearing in the bilinear form  $A_h(\cdot,\cdot)$ generally do not vanish, since $R_h(\boldsymbol{u}_h)$ and $R_h(\boldsymbol{v}_h)$ are not constant on interface edges.
\end{remark}

\begin{remark}
The proposed IFE method \eqref{IFE_method} is a modification of the scheme introduced in \cite{2021ji_IFE_stoke}.
In contrast to the earlier approach, the present formulation employs a divergence-conforming velocity reconstruction to modify the right-hand side and eliminates the pressure penalty term $J_h(\cdot,\cdot)$ defined in \cite{2021ji_IFE_stoke}.
We will show in Section~\ref{sec_LBB} that the inf-sup condition remains valid even without the pressure penalty term.
\end{remark}

\section{Properties of the IFE Space}\label{sec_pro}
In this section, we review several key properties of the IFE space established in \cite{2021ji_IFE_stoke}, which will be used in the subsequent analysis.
We begin by recalling a standard extension result (see, e.g., \cite{Gilbargbook}).
Suppose $v^\pm\in H^m(\Omega^\pm)$ with $m>0$. Then there exist extensions $ v_E^\pm \in H^m(\Omega)$ such that
\begin{equation}\label{extension}
v_E^\pm|_{\Omega^\pm}=v^\pm~\mbox{ and }~\|v_E^\pm\|_{H^m(\Omega)}\leq C\|v^\pm\|_{H^m(\Omega^\pm)},
\end{equation}
where the constant $C>0$ depends only on $\Omega^\pm$. 
This extension result also applies componentwise to vector-valued functions. In the sequel, we denote by $\boldsymbol{v}_E^\pm$ the extension of a vector function $\boldsymbol{v}^\pm\in H^m(\Omega^\pm)^2$. 

We also make use of polynomial extensions. For any $v^\pm \in P_k(\cup T_h^\pm)$, we slightly abuse notation by using the same symbol $v^\pm$ to denote its extension to a global polynomial in $P_k(\mathbb{R}^2)$.
The same convention applies to vector- and matrix-valued polynomial functions.

Define the IFE interpolation operator $\boldsymbol{\pi}_{h}^{\rm IFE}: \boldsymbol{V}+ \boldsymbol{X}_h^{\rm IFE} \rightarrow \boldsymbol{X}_h^{\rm IFE}$ and the standard CR interpolation operator $\boldsymbol{\pi}_{h}^{\rm CR}:  \boldsymbol{V}+ \boldsymbol{X}_h^{\rm IFE} \rightarrow \boldsymbol{X}_h$ by
\begin{equation*}
\int_e (\boldsymbol{\pi}_{h}^{\rm IFE}\boldsymbol{v})|_T=\int_e\{\boldsymbol{v}\},\quad \int_e(\boldsymbol{\pi}_{h}^{\rm CR}\boldsymbol{v})|_T=\int_e\{\boldsymbol{v}\},\quad \forall\, e\in\mathcal{E}_h,\quad \forall\, T\in\mathcal{T}_h^e.
\end{equation*}
Let $$P_0(\mathcal{T}_h)=\{q\in L^2(\Omega) : q|_T\in P_0(T)~~ \forall\, T\in\mathcal{T}_h\}.$$
Define the $L^2$-projection $\pi_h^0: L^2(\Omega)\rightarrow P_0(\mathcal{T}_h)$  by
\begin{equation*}
(\pi_h^0q)|_T =\frac{1}{|T|}\int_{T}q \qquad \forall\, T\in\mathcal{T}_h.
\end{equation*}
Obviously, if $q\in Q$, then  $\pi_h^0q\in Y_h$. Moreover, from \eqref{pro_Rh}, we obtain the following important property:
\[
\pi_h^0R_h(\boldsymbol{v}_h)=0\qquad\forall\, \boldsymbol{v}_h\in\boldsymbol{X}_h^{\rm IFE}.
\]

For any $(\boldsymbol{v},q)\in  \widetilde{\boldsymbol{H}^2H^1}\cap (\boldsymbol{V}\times Q)$,  the decomposition \eqref{decomp_dis} indicates that $\boldsymbol{\pi}_{h}^{\rm IFE }(\boldsymbol{v})$ serves as an approximation of the velocity $\boldsymbol{v}$, and 
$R_h(\,\boldsymbol{\pi}_{h}^{\rm IFE }(\boldsymbol{v})\,)+\pi_h^0q$ approximates the pressure $q$. 
The optimal approximation properties of these interpolants are stated in the following theorem (see Theorem 4.14 in \cite{2021ji_IFE_stoke}). 
%
\begin{theorem}\label{lem_interIFE0}
For any $(\boldsymbol{v},q)\in \widetilde{\boldsymbol{H}^2H^1}\cap (\boldsymbol{V}\times Q)$, there exists a positive constant $C$ independent of $h$ and the position of the interface relative to the mesh such that
\begin{align}
&\sum_{T\in\mathcal{T}_h^\Gamma}|\boldsymbol{v}_E^\pm-(\boldsymbol{\pi}_{h}^{\rm IFE }(\boldsymbol{v}))^\pm|^2_{H^m(T)}\leq Ch^{4-2m}(\|\boldsymbol{v}\|^2_{H^2(\cup \Omega^\pm)}+\|q\|^2_{H^1(\cup \Omega^\pm)}),~~m=0,1,\label{ultimate_ve0}\\
&\sum_{T\in\mathcal{T}_h^\Gamma}\|q_E^\pm-(R_h(\,\boldsymbol{\pi}_{h}^{\rm IFE }(\boldsymbol{v})\,)+\pi_h^0q)^\pm\|^2_{L^2(T)}\leq Ch^{2}(\|\boldsymbol{v}\|^2_{H^2(\cup \Omega^\pm)}+\|q\|^2_{H^1(\cup \Omega^\pm)}).\label{ultimate_pr0}
\end{align}
\end{theorem}
Taking into account the error resulting from the mismatch between $\Gamma$ and $\Gamma_h$, we also have the following result (see Theorem 4.15 in \cite{2021ji_IFE_stoke}).
\begin{theorem}\label{lem_interIFE1}
For any $(\boldsymbol{v},q)\in \widetilde{\boldsymbol{H}^2H^1}\cap (\boldsymbol{V}\times Q)$, there exists a positive constant $C$ independent of $h$ and the position of the interface relative to the mesh such that
\begin{align}
&\sum_{T\in\mathcal{T}_h}|\boldsymbol{v}-\boldsymbol{\pi}_{h}^{\rm IFE }(\boldsymbol{v})|^2_{H^m(T)}\leq Ch^{4-2m}(\|\boldsymbol{v}\|^2_{H^2(\cup \Omega^\pm)}+\|q\|^2_{H^1(\cup \Omega^\pm)}),~~m=0,1,\label{ultimate_ve1}\\
&\|q-(R_h(\,\boldsymbol{\pi}_{h}^{\rm IFE }(\boldsymbol{v})\,)+\pi_h^0q)\|^2_{L^2(\Omega)}\leq Ch^{2}(\|\boldsymbol{v}\|^2_{H^2(\cup \Omega^\pm)}+\|q\|^2_{H^1(\cup \Omega^\pm)}).\label{ultimate_pr1}
\end{align}
\end{theorem}

In the above results, $(\boldsymbol{v},q)$ is treated as a coupled vector-valued function, leading to the appearance of the term $\|q\|_{H^1(\cup\Omega^\pm)}$ on the right-hand sides of the approximation estimates.
However, since it has been established that velocity interpolation $\boldsymbol{\pi}_{h}^{\rm IFE }(\boldsymbol{v})$ depends solely on $\boldsymbol{v}$, but is independent of $q$, the presence of $\|q\|_{H^1(\cup\Omega^\pm)}$ in these bounds seems unnecessary.
In what follows, we demonstrate that this term can be removed from the right-hand sides of  \eqref{ultimate_ve0} and \eqref{ultimate_ve1}.
\begin{corollary}
For any $\boldsymbol{v}\in \boldsymbol{X}$, there exists a positive constant $C$ independent of $h$ and the position of the interface relative to the mesh such that for $m=0,1$,
\begin{equation}\label{ultimate_ve2}
\sum_{T\in\mathcal{T}_h^\Gamma}|\boldsymbol{v}_E^\pm-(\boldsymbol{\pi}_{h}^{\rm IFE }(\boldsymbol{v}))^\pm|^2_{H^m(T)}+\sum_{T\in\mathcal{T}_h}|\boldsymbol{v}-\boldsymbol{\pi}_{h}^{\rm IFE }(\boldsymbol{v})|^2_{H^m(T)}\leq Ch^{4-2m}\|\boldsymbol{v}\|^2_{H^2(\cup \Omega^\pm)}.
\end{equation}
\end{corollary}
\begin{proof}
By definition, we have $(\boldsymbol{v},\tilde{q})\in \widetilde{\boldsymbol{H}^2H^1}\cap (\boldsymbol{V}\times Q)$ with $\tilde{q}=R(\boldsymbol{v})$.  Applying estimates \eqref{ultimate_ve0} and \eqref{ultimate_ve1} with $q$ replaced by $\tilde{q}$,  and using \eqref{cons_qprime},  we obtain the desired result \eqref{ultimate_ve2}.
\end{proof}

Next, we introduce an IFE trace inequality that plays a key role in establishing the coercivity of the bilinear form $a_h(\cdot,\cdot)$.
For each interface element $T\in\mathcal{T}_h^\Gamma$, and for all $\boldsymbol{v}_h \in \boldsymbol{X}_h^{\rm IFE}$,  the condition $[\boldsymbol{v}_h]_{\Gamma_{h,T}}=\boldsymbol{0}$ implies that  $\boldsymbol{v}_h\in H^1(T)^2$. Consequently, the standard trace inequality holds:
\begin{equation}\label{trace_IFE0}
\| \boldsymbol{v}_h\|_{L^2(\partial T)}\leq C(h_T^{-1/2}\|\boldsymbol{v}_h\|_{L^2(T)}+h_T^{1/2} \|\nabla \boldsymbol{v}_h\|_{L^2(T)}) \qquad \forall\, \boldsymbol{v}_h \in \boldsymbol{X}_h^{\rm IFE}.
\end{equation}
However, this standard  trace inequality cannot be directly applied to $\nabla \boldsymbol{v}_h$  since $\boldsymbol{v}_h$ does not necessarily belong to $H^2(T)^2$.  The following lemma establishes a trace inequality for IFE functions (see Lemma 5.2 in \cite{2021ji_IFE_stoke}).
\begin{lemma}\label{trac_IFE}
For any interface element $T\in\mathcal{T}_h^\Gamma$,  there exists a constant $C>0$ independent of $h_T$ and the position of the interface relative to the mesh such that 
\begin{equation}\label{trac_IFE_inequality}
\|\nabla \boldsymbol{v}_h\|_{L^2(\partial T)}\leq Ch_T^{-1/2}\|\nabla \boldsymbol{v}_h\|_{L^2(T)}\quad \forall\, \boldsymbol{v}_h \in\boldsymbol{X}_h^{\rm IFE}.
\end{equation}
\end{lemma}

We next state the stability of the IFE interpolation operator (see Lemma 5.7 in \cite{2021ji_IFE_stoke}).
\begin{lemma}
For any $\boldsymbol{v}\in \boldsymbol{V}$, there exists a constant $C>0$ independent of $h$ and the position of the interface relative to the mesh such that for all $T\in\mathcal{T}_h^\Gamma$,
\begin{equation}\label{sta_Pi}
|\boldsymbol{\pi}_h^{\rm IFE}\boldsymbol{v}|_{H^1(T)}\leq C|\boldsymbol{v}|_{H^1(T)}.
\end{equation}
\end{lemma}

We also have the following estimate for the coupled IFE velocity and pressure.
\begin{lemma}\label{lem_jumpe}
Let $R_h$ be defined as in \eqref{def_Rh}.
For each element $T\in\mathcal{T}_h^{\Gamma}$ and each edge $e$ of the element $T$, there exists a constant $C>0$ independent of $h$ and the position of the interface  relative to the mesh such that
\begin{equation}\label{jumpe_tpre1}
\|R_h(\boldsymbol{v}_h)\|^2_{L^2(T)}+h_T\|R_h(\boldsymbol{v}_h)\|^2_{L^2(e)}\leq C|\boldsymbol{v}_h|^2_{H^1(T)}\qquad \forall\, \boldsymbol{v}_h \in\boldsymbol{X}_h^{\rm IFE}.
\end{equation}
\end{lemma}
\begin{proof}
On each interface element $T\in\mathcal{T}_h^\Gamma$, by Lemma~\ref{lem_IFEbasis} and the definition of $R_h$ in \eqref{def_Rh},  we have
\begin{equation*}
R_h(\boldsymbol{v}_h)=c_T(z_T-\pi_{h,T}^0z_T) \quad\mbox{ with }\quad c_T=\boldsymbol{n}_h^\top\boldsymbol{\sigma}(\mu^--\mu^+,\boldsymbol{\pi}_{h}^{\rm CR}\boldsymbol{v}_h, 0)\boldsymbol{n}_h.
\end{equation*}
From \eqref{def_zw}, it follows that $\|z_T\|_{L^\infty(T)}+\|\pi_{h,T}^0z_T\|_{L^\infty(T)}\leq 2$. 
Therefore, by the shape regularity of the mesh and the stability of the standard CR interpolation operator, we obtain
\begin{equation*}
\|R_h(\boldsymbol{v}_h)\|^2_{L^2(T)}+h_T\|R_h(\boldsymbol{v}_h)\|^2_{L^2(e)}\leq C |\boldsymbol{\pi}_h^{\rm CR}\boldsymbol{v}_h|^2_{H^1(T)} \leq C|\boldsymbol{v}_h|^2_{H^1(T)},
\end{equation*}
which yields the desired result.
\end{proof}

For any $\boldsymbol{v}_h \in \boldsymbol{X}_h^{\rm IFE}$, we define the broken divergence $\nabla_h \cdot \boldsymbol{v}_h$ and broken gradient $\nabla_h \boldsymbol{v}_h$ by
\begin{equation*}
(\nabla_h \cdot \boldsymbol{v}_h)|_T = \nabla \cdot (\boldsymbol{v}_h|_T), \qquad
(\nabla_h \boldsymbol{v}_h)|_T = \nabla (\boldsymbol{v}_h|_T), \qquad \forall\, T \in \mathcal{T}_h.
\end{equation*}
These definitions are well defined, since by \eqref{dis_jp2} we have $\boldsymbol{v}_h|_T \in H^1(T)^2$ for all $T \in \mathcal{T}_h$.

For the standard CR and RT interpolation operators, the following commutative diagram property holds (see, e.g., \cite{LINKE2014782}):
\begin{align}
\nabla_h \cdot \boldsymbol{\pi}_{h}^{\rm CR} \boldsymbol{v} &= \pi_h^0(\nabla_h \cdot \boldsymbol{v}) \qquad \forall\, \boldsymbol{v} \in \boldsymbol{V} + \boldsymbol{X}_h^{\rm IFE},\label{commu_CR}\\
\nabla \cdot \boldsymbol{\pi}_{h}^{\rm RT} \boldsymbol{v} &= \pi_h^0(\nabla_h \cdot \boldsymbol{v}) \qquad \forall\, \boldsymbol{v} \in \boldsymbol{V} + \boldsymbol{X}_h^{\rm IFE}. \label{commu_RT}
\end{align}

We show that this property also extends to the IFE interpolation operator.
\begin{lemma}[Commutative diagram property]\label{lem_commu}
For all $\boldsymbol{v} \in \boldsymbol{V}+ \boldsymbol{X}_h^{\rm IFE}$, we have
\begin{equation}\label{commu0}
\nabla_h \cdot \boldsymbol{\pi}_{h}^{\rm IFE} \boldsymbol{v} = \pi_h^0(\nabla_h \cdot \boldsymbol{v}).
\end{equation}
In particular, for the exact velocity $\boldsymbol{u}$ of problem \eqref{weakform0}, it holds that
\begin{equation}\label{commu1}
\nabla_h \cdot \boldsymbol{\pi}_{h}^{\rm IFE} \boldsymbol{u} = \pi_h^0(\nabla \cdot \boldsymbol{u}) = 0.
\end{equation}
\end{lemma}
\begin{proof}
By \eqref{dis_jp2} and \eqref{dis_jp3}, we see that $(\boldsymbol{\pi}_{h}^{\rm IFE} \boldsymbol{v})|_T\in H^1(T)^2$  and $\nabla \cdot (\boldsymbol{\pi}_{h}^{\rm IFE} \boldsymbol{v})|_T \in P_0(T)$ for all $T \in \mathcal{T}_h$. Applying the divergence theorem, together with the definition of $\boldsymbol{\pi}_{h}^{\rm IFE}$ and the relation $\int_e \boldsymbol{v} = 0$ for all boundary edges $e\in\mathcal{E}_h^\partial$, we obtain
\[
\int_{\Omega} \nabla_h \cdot \boldsymbol{\pi}_{h}^{\rm IFE} \boldsymbol{v} 
= \sum_{T \in \mathcal{T}_h} \int_T \nabla \cdot \boldsymbol{\pi}_{h}^{\rm IFE} \boldsymbol{v} 
= \sum_{e \in \mathcal{E}_h} \int_e [\boldsymbol{\pi}_{h}^{\rm IFE} \boldsymbol{v}] \cdot \boldsymbol{n}_e = 0.
\]
Therefore, we have
\[
\nabla_h \cdot \boldsymbol{\pi}_{h}^{\rm IFE} \boldsymbol{v} \in Y_h.
\]
The desired result \eqref{commu0} then follows directly from the identity
\[
\int_{\Omega} \nabla_h \cdot (\boldsymbol{\pi}_{h}^{\rm IFE} \boldsymbol{v} - \boldsymbol{v}) \, \hat{q}_h 
= \sum_{T \in \mathcal{T}_h} \int_{\partial T} (\boldsymbol{\pi}_{h}^{\rm IFE} \boldsymbol{v} - \boldsymbol{v}) \cdot \boldsymbol{n}_{\partial T} \, \hat{q}_h = 0 
\qquad \forall \, \hat{q}_h \in Y_h,
\]
where we have used the definition of $\boldsymbol{\pi}_{h}^{\rm IFE}$ and the fact that $\hat{q}_h$ is constant on each element. 
The identity \eqref{commu1} follows directly from \eqref{commu0} since $\nabla \cdot \boldsymbol{u} = 0$. 
\end{proof}

\section{Analysis of the IFE Method}\label{sec_analy}
The pressure-robust analysis is based on reformulating the discrete scheme as an elliptic problem (see \cite{LINKE2014782}). To this end, we introduce the space of discretely divergence-free functions:
\begin{equation}\label{def_Zh}
\boldsymbol{Z}_h = \{ \boldsymbol{v}_h \in \boldsymbol{X}_h^{\rm IFE} : b_h(\boldsymbol{v}_h, \hat{q}_h) = 0 \quad \forall\, \hat{q}_h \in Y_h \}.
\end{equation}
From \eqref{commu1}, it follows that
\begin{equation}\label{zh_pi}
\boldsymbol{\pi}_{h}^{\rm IFE} \boldsymbol{u} \in \boldsymbol{Z}_h.
\end{equation}
Moreover, by the construction of the IFE functions in \eqref{dis_jp0}, we have $\nabla_h \cdot \boldsymbol{v}_h \in Y_h$ for all $\boldsymbol{v}_h \in \boldsymbol{X}_h^{\rm IFE}$. Consequently,
\begin{equation}\label{w_h_eq_0}
\nabla_h \cdot \boldsymbol{v}_h = 0 \qquad \forall\, \boldsymbol{v}_h \in \boldsymbol{Z}_h.
\end{equation}
In addition, by \eqref{commu_RT}, we have
\begin{equation}\label{w_h_RT}
\nabla \cdot \boldsymbol{\pi}_{h}^{\rm RT} \boldsymbol{v}_h = 0 \qquad \forall\, \boldsymbol{v}_h \in \boldsymbol{Z}_h.
\end{equation}
In other words, the discrete velocity \(\boldsymbol{v}_h\) is elementwise divergence-free but is not
\(H(\operatorname{div})\)-conforming.
In contrast, the reconstructed velocity $\boldsymbol{\pi}_{h}^{\rm RT} \boldsymbol{v}_h$ is globally
$H(\operatorname{div})$-conforming and pointwise divergence-free.

The discrete problem~\eqref{IFE_method} can be rewritten in the following elliptic form: seek $\boldsymbol{u}_h\in \boldsymbol{Z}_h$ such that 
\begin{equation}\label{dis_prob_Zh}
A_h(\,(\boldsymbol{u}_h,R_h(\boldsymbol{u}_h)), (\boldsymbol{v}_h,R_h(\boldsymbol{v}_h))\, )=l_h(\boldsymbol{v}_h) \quad \forall \,  \boldsymbol{v}_h\in \boldsymbol{Z}_h.
\end{equation}

To derive the error estimates, we define the following mesh-dependent norms for all $(\boldsymbol{v}, q) \in \widetilde{\boldsymbol{H}^2H^1}\cap (\boldsymbol{V}\times Q)+\widetilde{\boldsymbol{V}Q}_{h}^{\rm{IFE}}$:
\begin{equation}\label{def_norms}
\begin{aligned}
&\|\boldsymbol{v} \|^2_{1,h}=\sum_{T\in\mathcal{T}_h} | \boldsymbol{v}|^2_{H^1(T)},~~\interleave \boldsymbol{v} \interleave^2_{1,h}=\sum_{T\in\mathcal{T}_h} \|\sqrt{2\mu_h} \boldsymbol{\epsilon}(\boldsymbol{v})\|^2_{L^2(T)}+\sum_{e\in\mathcal{E}_h}\frac{\mu_{\rm max}}{h_e}\left\|[\boldsymbol{v}] \right\|_{L^2(e)}^2,\\
&\interleave\boldsymbol{v} \interleave^2_{*,h}=\interleave \boldsymbol{v} \interleave^2_{1,h}+\sum_{e\in\mathcal{E}_h^\Gamma}h_e\|\{ \sqrt{2\mu_h}\boldsymbol{\epsilon}(\boldsymbol{v})\boldsymbol{n}_e\} \|_{L^2(e)}^2 +\sum_{e\in\mathcal{E}_h^\Gamma}\frac{\eta\mu_{\rm max}}{h_e}\|[\boldsymbol{v}] \|_{L^2(e)}^2,\\
&\|q\|^2_{*,pre}=\|q\|_{L^2(\Omega)}^2+\sum_{e\in\mathcal{E}_h^\Gamma}h_e\|\{q\}\|^2_{L^2(e)},\qquad \|(\boldsymbol{v}, q)\|^2_*= \interleave \boldsymbol{v} \interleave^2_{*,h}+\|q\|^2_{*,pre}.
\end{aligned}
\end{equation}
Since $\boldsymbol{v}|_{T}\in H^1(T)^2$ for all $T\in\mathcal{T}_h$, and $\int_e[\boldsymbol{v}]=\boldsymbol{0}$ for all $e\in\mathcal{E}_h$, the Poincar\'e--Friedrichs inequality and the Korn inequality  for piecewise $H^1$ functions imply (see \cite{brenner2003poincare, brenner2004korn})
\begin{equation}\label{korn_ineq}
\|\boldsymbol{v}\|^2_{L^2(\Omega)}\leq  C\sum_{T\in\mathcal{T}_h}|\boldsymbol{v}|^2_{H^1(T)}\leq C\sum_{T\in\mathcal{T}_h}\|\boldsymbol{\epsilon}(\boldsymbol{v})\|^2_{L^2(T)}+C\sum_{e\in\mathcal{E}_h}\frac{1}{h_e}\|[\boldsymbol{v}]\|_{L^2(e)}^2.
\end{equation}
Hence, $\|\cdot\|_{1,h}$ and $\interleave\cdot \interleave_{1,h}$ are indeed norms on the space $\boldsymbol{X}+\boldsymbol{X}_h^{\rm IFE}$ and the following relation holds
\begin{equation}\label{norm_rela}
\|\boldsymbol{v}\|_{1,h}\leq C \interleave\boldsymbol{v} \interleave_{1,h}\qquad \forall \boldsymbol{v}\in\boldsymbol{X}+\boldsymbol{X}_h^{\rm IFE}.
\end{equation}

By the definition \eqref{def_norms}, and using the estimates \eqref{trace_IFE0}, \eqref{trac_IFE_inequality}, \eqref{jumpe_tpre1}, and \eqref{korn_ineq}, we establish the following norm equivalence for IFE functions.
\begin{lemma}
There exist constants $C_1, C_2>0$ independent of $h$ and the position of the interface  relative to the mesh such that
\begin{equation}\label{norm_equal}
C_1 \|\boldsymbol{v}_h\|_{1,h} \leq \|(\boldsymbol{v}_h, R_h(\boldsymbol{v}_h))\|_*\leq C_2\|\boldsymbol{v}_h\|_{1,h}\qquad \forall\, \boldsymbol{v}_h\in \boldsymbol{X}_h^{\rm IFE}.
\end{equation}
\end{lemma}

Leveraging the optimal approximation properties of the IFE space established in Section~\ref{sec_pro}, we obtain the following estimates for the interpolation error measured in the norms defined in \eqref{def_norms}.
\begin{lemma}\label{ener_app}
For any $\boldsymbol{v}\in \boldsymbol{X}$, there exists a positive constant $C$ independent of $h$ and the position of the interface relative to the mesh such that
\begin{align}
\|\boldsymbol{v}-\boldsymbol{\pi}_h^{\rm IFE} \boldsymbol{v} \|_{*,h}& \leq Ch\| \boldsymbol{v}\|_{H^2(\cup \Omega^\pm)},\label{iter_err_new1}\\
\|R(\boldsymbol{v})-\pi_{h}^0 R(\boldsymbol{v})-R_h(\boldsymbol{\pi}_h^{\rm IFE}\boldsymbol{v})\|_{*,pre}& \leq Ch\| \boldsymbol{v}\|_{H^2(\cup \Omega^\pm)}.\label{iter_err_new2}
\end{align}
\end{lemma}
\begin{proof}
It suffices to consider interface edges, since the corresponding estimate for non-interface edges is standard.
For each interface edge $e\in\mathcal{E}_h^\Gamma$, let $e^\pm=e\cap \Omega^\pm$ and $e_h^\pm=e\cap \Omega_h^\pm$.  Using the fact that $e^\pm =e_h^\pm$,  we have
\begin{equation*}
\begin{aligned}
\|\{ \sqrt{2\mu_h}\boldsymbol{\epsilon}( \boldsymbol{v}-\boldsymbol{\pi}_h^{\rm IFE} \boldsymbol{v}) \boldsymbol{n}_e\}\|_{L^2(e)}^2 &= \sum_{s=\pm}\|\{ \sqrt{2\mu_h}\boldsymbol{\epsilon}( \boldsymbol{v}_E^s-(\boldsymbol{\pi}_h^{\rm IFE} \boldsymbol{v})^s) \boldsymbol{n}_e\} \|_{L^2(e^s)}^2\\
&\leq \sum_{s=\pm}\|\{ \sqrt{2\mu_h}\boldsymbol{\epsilon}( \boldsymbol{v}_E^s-(\boldsymbol{\pi}_h^{\rm IFE} \boldsymbol{v})^s) \boldsymbol{n}_e\} \|_{L^2(e)}^2,
\end{aligned}
\end{equation*}
which together with the standard trace inequality, the interpolation error estimate \eqref{ultimate_ve2}, and the extension result \eqref{extension}  yields 
\begin{equation*}
\begin{aligned}
\sum_{e\in\mathcal{E}_h^\Gamma}h_e\|\{ \sqrt{2\mu_h}\boldsymbol{\epsilon}( \boldsymbol{v}-\boldsymbol{\pi}_h^{\rm IFE} \boldsymbol{v}) \boldsymbol{n}_e\} \|_{L^2(e)}^2 &\leq C\sum_{T\in\mathcal{T}_h^\Gamma}\sum_{s=\pm}\left(| \boldsymbol{v}_E^s-(\boldsymbol{\pi}_h^{\rm IFE} \boldsymbol{v})^s|^2_{H^1(T)}+h_T^2| \boldsymbol{v}_E^s|^2_{H^2(T)}\right)\\
&\leq Ch^2\|\boldsymbol{v}\|^2_{H^2(\cup \Omega^\pm)}.
\end{aligned}
\end{equation*}
Analogously, we obtain
\begin{equation*}
\begin{aligned}
\sum_{e\in\mathcal{E}_h^\Gamma}\frac{1}{h_e}\|[ \boldsymbol{v}-\boldsymbol{\pi}_h^{\rm IFE} \boldsymbol{v}] \|_{L^2(e)}^2\leq Ch^2\|\boldsymbol{v}\|^2_{H^2(\cup \Omega^\pm)}.
\end{aligned}
\end{equation*}
Combining the above results with \eqref{ultimate_ve2} yields the desired estimate \eqref{iter_err_new1}.

Let $\tilde{q}=R(\boldsymbol{v})$. We have $(\boldsymbol{v},\tilde{q})\in \widetilde{\boldsymbol{H}^2H^1}\cap (\boldsymbol{V}\times Q)$.  Similar to the proof of  \eqref{iter_err_new1}, we apply the standard trace inequality and make use of estimates \eqref{ultimate_pr0} and \eqref{ultimate_pr1} to get
\[
\|R(\boldsymbol{v})-\pi_{h}^0 R(\boldsymbol{v})-R_h(\boldsymbol{\pi}_h^{\rm IFE}\boldsymbol{v})\|_{*,pre}\leq Ch(\|\boldsymbol{v}\|_{H^2(\cup \Omega^\pm)}+\|\tilde{q}\|_{H^1(\cup \Omega^\pm)}),
\]
which together with \eqref{cons_qprime}  yields the desired estimate \eqref{iter_err_new2}.
 \end{proof}

\subsection{Boundedness and Coercivity}
Using \eqref{def_Ah}, \eqref{korn_ineq}, and the Cauchy--Schwarz inequality, we obtain the following boundedness result.
\begin{lemma}\label{lem_bd_Ah}
For all $(\boldsymbol{u}, p)$ and $ (\boldsymbol{v}, q)$ in  $\widetilde{\boldsymbol{H}^2H^1}+\widetilde{\boldsymbol{V}Q}_{h}^{\rm{IFE}}$, it holds 
\begin{equation}\label{bd_Ah}
|A_h(\,(\boldsymbol{u},p),(\boldsymbol{v},q)\,)|\leq C_{A}\|(\boldsymbol{u},p)\|_*\|(\boldsymbol{v},q)\|_*,
\end{equation}
where $C_{A}$ is independent of $h$ and the position of the interface relative to the mesh.
\end{lemma}
Noting that
$$A_h(\,(\boldsymbol{v}_h, R_h(\boldsymbol{v}_h)), (\boldsymbol{v}_h, R_h(\boldsymbol{v}_h))\,)=a_h(\boldsymbol{v}_h,\boldsymbol{v}_h)\qquad \forall\, \boldsymbol{v}_h\in\boldsymbol{X}_h^{\rm IFE},$$
we establish the following coercivity result on the IFE space (see Lemma 5.3 of \cite{2021ji_IFE_stoke}).
\begin{lemma}\label{lem_cor}
There exists a constant $C_{a}>0$ independent of $h$ and the position of the interface relative to the mesh such that 
\begin{equation}\label{inequ_cor}
A_h(\,(\boldsymbol{v}_h, R_h(\boldsymbol{v}_h)), (\boldsymbol{v}_h, R_h(\boldsymbol{v}_h))\,)\geq C_{a}\|\boldsymbol{v}_h\|^2_{1,h}\qquad \forall\, \boldsymbol{v}_h\in\boldsymbol{X}_h^{\rm IFE}
\end{equation}
holds for  $\theta=-1$ with arbitrary $\eta\geq 0$, and for $\theta=1$ provided that $\eta$ is sufficiently large.
\end{lemma}

\subsection{Discrete Inf--Sup Condition}\label{sec_LBB}
\begin{lemma}\label{lem_inf_sup}
There exists a constant $\beta > 0$, independent of $h$ and the position of the interface relative to the mesh, such that 
\begin{equation}\label{infsup1}
 \sup_{\boldsymbol{v}_h \in \boldsymbol{X}_h^{\rm IFE}} \frac{b_h(\boldsymbol{v}_h, \hat{q}_h)}{\| \boldsymbol{v}_h \|_{1,h}} \geq \beta \|\hat{q}_h\|_{L^2(\Omega)} \qquad \forall\, \hat{q}_h \in Y_h.
\end{equation}
\end{lemma}
\begin{proof}
We employ the Fortin argument to establish the discrete inf-sup condition.
In view of Remark~\ref{remark_bh} and Lemma~\ref{lem_commu}, the IFE interpolation operator $\boldsymbol{\pi}_{h}^{\rm IFE} : \boldsymbol{V} \rightarrow \boldsymbol{X}_h^{\rm IFE}$ satisfies
\[
b_h(\boldsymbol{\pi}_{h}^{\rm IFE} \boldsymbol{v} - \boldsymbol{v}, \hat{q}_h) = 0 \qquad \forall\, \hat{q}_h \in Y_h.
\]
Moreover, from the stability estimate \eqref{sta_Pi} for interface elements and the standard stability of the CR interpolation operator on non-interface elements, we have 
\[
\|\boldsymbol{\pi}_h^{\rm IFE} \boldsymbol{v}\|_{1,h} \leq C \|\boldsymbol{v}\|_{1,h} \qquad \forall\, \boldsymbol{v} \in \boldsymbol{V},
\]
which implies that $\boldsymbol{\pi}_h^{\rm IFE}$ is a Fortin operator.
Furthermore, since $b_h(\boldsymbol{v}, q) = b(\boldsymbol{v}, q)$ for all $(\boldsymbol{v}, q) \in \boldsymbol{V} \times Q$, and $Y_h \subset Q$, the continuous inf-sup condition applies.  
Therefore, the desired discrete inf-sup condition \eqref{infsup1} follows from the Fortin lemma (see, e.g., \cite{Brezzi}).
\end{proof}

As a consequence of \eqref{bd_Ah}-\eqref{infsup1}, the discrete problem ~\eqref{IFE_method} is well-posed (see, e.g., \cite{Brezzi}).

\subsection{Consistency}
In this section, we analyze the consistency error, which primarily arises from the interface approximation and the variational crime introduced by the velocity reconstruction. To this end, we first prove a technical lemma.
\begin{lemma}\label{lem_cons_0}
Assume that the exact velocity of problem~(\ref{weakform0}) satisfies $\boldsymbol{u} \in \boldsymbol{X}$. Then there exists a constant $C> 0$, independent of $h$ and the position of the interface relative to the mesh, such that
\begin{equation*}
\begin{aligned}
&\left|A_h(\,(\boldsymbol{u}, R(\boldsymbol{u})), (\boldsymbol{v}_h,R_h(\boldsymbol{v}_h))\,)-\sum_{s=\pm}\int_{\Omega^s}(-\nabla\cdot (2\mu\boldsymbol{\epsilon}(\boldsymbol{u})) +\nabla R(\boldsymbol{u}))\cdot \boldsymbol{\pi}_h^{\rm RT}\boldsymbol{v}_h\right|\\
&\qquad \qquad\leq Ch\|\boldsymbol{u}\|_{H^2(\cup \Omega^\pm)} \|\boldsymbol{v}_h\|_{1,h}\qquad \forall\,\boldsymbol{v}_h\in \boldsymbol{X}_h^{\rm IFE}.
\end{aligned}
\end{equation*}
\end{lemma}
\begin{proof}
By the triangle inequality, we decompose the left-hand side as follows:
\begin{equation}\label{pro_deco_Ah}
\begin{aligned}
&\left|A_h(\,(\boldsymbol{u}, R(\boldsymbol{u})), (\boldsymbol{v}_h,R_h(\boldsymbol{v}_h))\,)-\sum_{s=\pm}\int_{\Omega^s}(-\nabla\cdot (2\mu\boldsymbol{\epsilon}(\boldsymbol{u})) +\nabla R(\boldsymbol{u}))\cdot \boldsymbol{\pi}_h^{\rm RT}\boldsymbol{v}_h\right|\\
&\qquad\leq \left|A_h(\,(\boldsymbol{u}, R(\boldsymbol{u})), (\boldsymbol{v}_h,R_h(\boldsymbol{v}_h))\,)-\sum_{s=\pm}\int_{\Omega^s}(-\nabla\cdot (2\mu\boldsymbol{\epsilon}(\boldsymbol{u})) +\nabla R(\boldsymbol{u}))\cdot \boldsymbol{v}_h\right|\\
&\qquad\qquad+\left| \sum_{s=\pm}\int_{\Omega^s}(-\nabla\cdot (2\mu\boldsymbol{\epsilon}(\boldsymbol{u})) +\nabla R(\boldsymbol{u}))\cdot (\boldsymbol{\pi}_h^{\rm RT}\boldsymbol{v}_h-\boldsymbol{v}_h)\right|\\
&:=(\mathrm{I})_1+(\mathrm{I})_2.
\end{aligned}
\end{equation}
For the second term $(\mathrm{I})_2$, we apply the standard Raviart--Thomas interpolation error estimate (see Lemma 2 in \cite{LINKE2014782})
\[
\|\boldsymbol{\pi}_h^{\rm RT}\boldsymbol{v}_h-\boldsymbol{v}_h\|_{L^2(\Omega)}\leq Ch\|\boldsymbol{v}_h\|_{1,h}
\]
and the following estimate from \eqref{cons_qprime}
\begin{equation}\label{pro_RU}
\|R(\boldsymbol{u})\|_{H^1(\cup \Omega^\pm)}\leq C\|\boldsymbol{u}\|_{H^2(\cup\Omega^\pm)}
\end{equation}
 to obtain
\[
(\mathrm{I})_2\leq Ch\|\boldsymbol{u}\|_{H^2(\cup \Omega^\pm)} \|\boldsymbol{v}_h\|_{1,h}.
\]
For the first term $(\mathrm{I})_1$, noting that the pair $(\boldsymbol{u}, R(\boldsymbol{u}))$ satisfies the interface jump condition
\[
[\boldsymbol{\sigma}(\mu,\boldsymbol{u},R(\boldsymbol{u}) )\boldsymbol{n}]_\Gamma=\boldsymbol{0},
\]
we apply integration by parts to derive
\begin{equation}\label{pro_I2_0}
\begin{aligned}
\sum_{s=\pm}\int_{\Omega^s}(&-\nabla\cdot (2\mu\boldsymbol{\epsilon}(\boldsymbol{u})) +\nabla R(\boldsymbol{u}))\cdot \boldsymbol{v}_h\\
&=\sum_{T\in\mathcal{T}_h}\int_{T}2\mu\boldsymbol{\epsilon}(\boldsymbol{u}):\boldsymbol{\epsilon}(\boldsymbol{v}_h)-\sum_{e\in\mathcal{E}_h}\int_e \{ 2\mu \boldsymbol{\epsilon}(\boldsymbol{u})\boldsymbol{n}_e\}\cdot[\boldsymbol{v}_h]\\
&\quad-\sum_{T\in\mathcal{T}_h} \int_T R(\boldsymbol{u})\nabla\cdot \boldsymbol{v}_h+\sum_{e\in\mathcal{E}_h}\int_e\{R(\boldsymbol{u})\}[\boldsymbol{v}_h]\cdot\boldsymbol{n}_e.
\end{aligned}
\end{equation}
On the other hand, using $\nabla \cdot \boldsymbol{u} = 0$ and $[\boldsymbol{u}]_e = \boldsymbol{0}$ on each edge $e$, we deduce from \eqref{def_Ah} that
\[
b_h(\boldsymbol{u},R_h(\boldsymbol{v}_h))=0.
\]
Hence, we have
\[
\begin{aligned}
A_h(\,(\boldsymbol{u}, R(\boldsymbol{u})), (\boldsymbol{v}_h,R_h(\boldsymbol{v}_h))\,)
&=a_h(\boldsymbol{u},\boldsymbol{v}_h)+b_h(\boldsymbol{v}_h, R(\boldsymbol{u}))-b_h(\boldsymbol{u},R_h(\boldsymbol{v}_h))\\
&=a_h(\boldsymbol{u},\boldsymbol{v}_h)+b_h(\boldsymbol{v}_h, R(\boldsymbol{u}) ).
\end{aligned}
\]
Using $[\boldsymbol{u}]_e = \boldsymbol{0}$ again and applying \eqref{def_Ah}, we obtain
\begin{equation}\label{pro_I2_1}
\begin{aligned}
A_h(\,(\boldsymbol{u}, R(\boldsymbol{u})), (\boldsymbol{v}_h,R_h(\boldsymbol{v}_h))\,)&=\sum_{T\in\mathcal{T}_h}\int_{T}2\mu_h\boldsymbol{\epsilon}(\boldsymbol{u}):\boldsymbol{\epsilon}(\boldsymbol{v}_h)-\sum_{e\in\mathcal{E}_h^\Gamma}\int_e \{ 2\mu_h \boldsymbol{\epsilon}(\boldsymbol{u})\boldsymbol{n}_e\}\cdot[\boldsymbol{v}_h]\\
&\quad-\sum_{T\in\mathcal{T}_h} \int_T R(\boldsymbol{u})\nabla\cdot \boldsymbol{v}_h+\sum_{e\in\mathcal{E}_h^\Gamma}\int_e\{R(\boldsymbol{u})\}[\boldsymbol{v}_h]\cdot\boldsymbol{n}_e.
\end{aligned}
\end{equation}
Combining \eqref{pro_I2_0} and \eqref{pro_I2_1}, and noting that $\mu|_e = \mu_h|_e$ on each edge $e$, we further derive
\begin{equation}\label{pro_I2_2}
\begin{aligned}
(\mathrm{I})_1\leq &\left|\sum_{T\in\mathcal{T}_h}\int_{T^\triangle}2(\mu-\mu_h)\boldsymbol{\epsilon}(\boldsymbol{u}):\boldsymbol{\epsilon}(\boldsymbol{v}_h)\right|\\
&+\left|\sum_{e\in\mathcal{E}_h^{non}}\int_e -\{ 2\mu \boldsymbol{\epsilon}(\boldsymbol{u})\boldsymbol{n}_e\}\cdot[\boldsymbol{v}_h]+ \{R(\boldsymbol{u})\}[\boldsymbol{v}_h]\cdot\boldsymbol{n}_e\right|.
\end{aligned}
\end{equation}
By \eqref{triT_rela}, we may apply  Lemma~\ref{strip} with $\delta=Ch^2$ to estimate the first term:
\[
\left|\sum_{T\in\mathcal{T}_h}\int_{T^\triangle}2(\mu-\mu_h)\boldsymbol{\epsilon}(\boldsymbol{u}):\boldsymbol{\epsilon}(\boldsymbol{v}_h)\right|\leq Ch\|\boldsymbol{u}\|_{H^2(\cup\Omega^\pm)}\|\boldsymbol{v}_h\|_{1,h}.
\]
For the second term in \eqref{pro_I2_2}, since it involves only non-interface edges, standard arguments yield
\[
\left|\sum_{e\in\mathcal{E}_h^{non}}\int_e \{ 2\mu \boldsymbol{\epsilon}(\boldsymbol{u})\boldsymbol{n}_e\}\cdot[\boldsymbol{v}_h]+ \{R(\boldsymbol{u})\}[\boldsymbol{v}_h]\cdot\boldsymbol{n}_e\right|\leq Ch\|\boldsymbol{u}\|_{H^2(\cup\Omega^\pm)}\|\boldsymbol{v}_h\|_{1,h},
\]
where the inequality \eqref{pro_RU} is also used. Therefore, the estimate for term $(\mathrm{I})_1$ is completed as
\[
(\mathrm{I})_1\leq Ch\|\boldsymbol{u}\|_{H^2(\cup\Omega^\pm)}\|\boldsymbol{v}_h\|_{1,h}.
\]
This completes the proof.
\end{proof}

We are now in a position to derive the following consistency error estimate.
\begin{lemma}\label{lem_consis} Assume that the exact velocity of problem~(\ref{weakform0}) satisfies $\boldsymbol{u} \in \boldsymbol{X}$. Then, for the discrete scheme~\eqref{dis_prob_Zh}, the following consistency error estimate holds:
\[
\sup_{\boldsymbol{w}_h\in \boldsymbol{Z}_h}\frac{\left|A_h(\,(\boldsymbol{u}, R(\boldsymbol{u})), (\boldsymbol{w}_h,R_h(\boldsymbol{w}_h))\,)-l_h(\boldsymbol{w}_h)\right|}{\|\boldsymbol{w}_h\|_{1,h}}\leq Ch\|\boldsymbol{u}\|_{H^2(\cup\Omega^\pm)},
\] 
where the constant $C> 0$ is independent of $h$ and the position of the interface relative to the mesh.
\end{lemma}
\begin{proof}
Let $\hat{p} = p - R(\boldsymbol{u})$ and $\boldsymbol{w}_h\in \boldsymbol{Z}_h$. By \eqref{hat_q_b}, we have $\hat{p} \in H^1(\Omega)$, and hence
\[
\int_{\Omega} \nabla \hat{p}\cdot  \boldsymbol{\pi}_h^{\rm RT}\boldsymbol{w}_h=0,
\]
where we have used \eqref{w_h_RT} and the fact that $(\boldsymbol{\pi}_h^{\rm RT} \boldsymbol{w}_h\cdot \boldsymbol{n})|_{\partial\Omega} = \boldsymbol{0}$.
Therefore, we have
\begin{equation}\label{pro_lh_er}
\begin{aligned}
l_h(\boldsymbol{w}_h)&=\int_\Omega \boldsymbol{f}\cdot \boldsymbol{\pi}_h^{\rm RT}\boldsymbol{w}_h=\sum_{s=\pm}\int_{\Omega^s}(-\nabla\cdot (2\mu\boldsymbol{\epsilon}(\boldsymbol{u})) +\nabla p) \cdot \boldsymbol{\pi}_h^{\rm RT}\boldsymbol{w}_h\\
&=\sum_{s=\pm}\int_{\Omega^s}(-\nabla\cdot (2\mu\boldsymbol{\epsilon}(\boldsymbol{u})) +\nabla R(\boldsymbol{u})+\nabla \hat{p})\cdot \boldsymbol{\pi}_h^{\rm RT}\boldsymbol{w}_h\\
&=\sum_{s=\pm}\int_{\Omega^s}(-\nabla\cdot (2\mu\boldsymbol{\epsilon}(\boldsymbol{u})) +\nabla R(\boldsymbol{u}))\cdot \boldsymbol{\pi}_h^{\rm RT}\boldsymbol{w}_h.
\end{aligned}
\end{equation}
Now, applying Lemma~\ref{lem_cons_0}, we obtain
\[
|A_h(\,(\boldsymbol{u}, R(\boldsymbol{u})), (\boldsymbol{w}_h,R_h(\boldsymbol{w}_h))\,)-l_h(\boldsymbol{w}_h)|\leq Ch\|\boldsymbol{u}\|_{H^2(\cup \Omega^\pm)} \|\boldsymbol{w}_h\|_{1,h}.
\]
This completes the proof.
\end{proof}

\subsection{Error Estimates}
With the above preparations, we state the main error estimates as follows.
\begin{theorem}\label{theo_ulti}
Assume that the exact solution of problem~\eqref{weakform0} satisfies $(\boldsymbol{u}, p) \in \widetilde{\boldsymbol{H}^2H^1}$. 
Then the discrete solution $(\boldsymbol{u}_h, p_h)$ of scheme~\eqref{IFE_method}, with $p_h = R_h(\boldsymbol{u}_h) + \hat{p}_h$, satisfies the following error estimates:
\begin{align}
\|\boldsymbol{u}-\boldsymbol{u}_h\|_{1,h}&\leq C\|\boldsymbol{u}-\boldsymbol{u}_h\|_{*,h}\leq Ch\|\boldsymbol{u}\|_{H^2(\cup \Omega^\pm)},\label{error_1}\\
\|\pi_h^0 p-\hat{p}_h\|_{L^2(\Omega)}&\leq Ch\|\boldsymbol{u}\|_{H^2(\cup \Omega^\pm)},\label{error_2}\\
\|p-p_h\|_{L^2(\Omega)}&\leq Ch(\|\boldsymbol{u}\|_{H^2(\cup \Omega^\pm)}+\|p\|_{H^1(\cup \Omega^\pm)}),\label{error_3}
\end{align}
where  the constant $C > 0$ is independent of $h$ and the position of the interface relative to the mesh.
\end{theorem}
\begin{proof}
i) The analysis begins with \eqref{dis_prob_Zh}, where the trial and test functions lie in the discretely divergence-free space $\boldsymbol{Z}_h$.
For an arbitrary $\boldsymbol{v}_h\in \boldsymbol{Z}_h$, let $\boldsymbol{w}_h=\boldsymbol{u}_h-\boldsymbol{v}_h$.
Then $\boldsymbol{w}_h\in \boldsymbol{Z}_h$, and we obtain
\begin{equation}
\begin{aligned}
C_{a}\|\boldsymbol{w}_h\|_{1,h}^2&\leq A_h(\,(\boldsymbol{w}_h, R_h(\boldsymbol{w}_h)),(\boldsymbol{w}_h,R_h(\boldsymbol{w}_h))\,)\\
&= A_h(\,(\boldsymbol{u}_h-\boldsymbol{v}_h, R_h(\boldsymbol{u}_h)-R_h(\boldsymbol{v}_h)), (\boldsymbol{w}_h,R_h(\boldsymbol{w}_h))\,)\\
&= A_h(\,(\boldsymbol{u}-\boldsymbol{v}_h, R(\boldsymbol{u})-R_h(\boldsymbol{v}_h)),(\boldsymbol{w}_h,R_h(\boldsymbol{w}_h))\,)\\
&\quad+A_h(\,(\boldsymbol{u}_h-\boldsymbol{u}, R_h(\boldsymbol{u}_h)-R(\boldsymbol{u}) ), (\boldsymbol{w}_h,R_h(\boldsymbol{w}_h))\,).
\end{aligned}
\end{equation}
Since $R(\boldsymbol{u})\in L_0^2(\Omega)$, we have $\pi_{h}^0 R(\boldsymbol{u})\in Y_h$. Moreover, because $\boldsymbol{w}_h\in \boldsymbol{Z}_h$, using \eqref{def_Ah} and  \eqref{def_Zh}, it follows that
\begin{equation}\label{pro_bh_pi}
A_h(\,(\boldsymbol{0}, \pi_{h}^0 R(\boldsymbol{u})),(\boldsymbol{w}_h,R_h(\boldsymbol{w}_h))\,)=b_h(\boldsymbol{w}_h, \pi_{h}^0 R(\boldsymbol{u}))=0.
\end{equation}
Therefore, we have
\begin{equation}
C_{a}\|\boldsymbol{w}_h\|_{1,h}^2\leq (\mathrm{II})_1+(\mathrm{II})_2,
\end{equation}
with
\[
\begin{aligned}
(\mathrm{II})_1&:= \left| A_h(\,(\boldsymbol{u}-\boldsymbol{v}_h, R(\boldsymbol{u})-\pi_{h}^0 R(\boldsymbol{u})-R_h(\boldsymbol{v}_h)),\,(\boldsymbol{w}_h, R_h(\boldsymbol{w}_h))\,) \right|,\\
(\mathrm{II})_2&:=\left|A_h(\,(\boldsymbol{u}_h-\boldsymbol{u}, R_h(\boldsymbol{u}_h)-R(\boldsymbol{u}) ), \, (\boldsymbol{w}_h,R_h(\boldsymbol{w}_h))\,)\right|.
\end{aligned}
\]
For the first term $(\mathrm{II})_1$, by \eqref{bd_Ah}, \eqref{norm_equal} and \eqref{def_norms}, we obtain
\[
\begin{aligned}
(\mathrm{II})_1&\leq C_A \|(\boldsymbol{u}-\boldsymbol{v}_h, R(\boldsymbol{u})-\pi_{h}^0 R(\boldsymbol{u})-R_h(\boldsymbol{v}_h))\|_*\,\|(\boldsymbol{w}_h,R_h(\boldsymbol{w}_h))\|_*\\
&\leq C_AC_2\|(\boldsymbol{u}-\boldsymbol{v}_h, R(\boldsymbol{u})-\pi_{h}^0 R(\boldsymbol{u})-R_h(\boldsymbol{v}_h))\|_*\,\|\boldsymbol{w}_h\|_{1,h}\\
&\leq C_AC_2\left(\interleave \boldsymbol{u}-\boldsymbol{v}_h  \interleave_{*,h}+\|R(\boldsymbol{u})-\pi_{h}^0 R(\boldsymbol{u})-R_h(\boldsymbol{v}_h)\|_{*,pre}\right)\|\boldsymbol{w}_h\|_{1,h}.
\end{aligned}
\]
For the second term $(\mathrm{II})_2$, we use the discrete problem \eqref{dis_prob_Zh}  to get
\[
(\mathrm{II})_2=|A_h(\,(\boldsymbol{u}, R(\boldsymbol{u})), (\boldsymbol{w}_h,R_h(\boldsymbol{w}_h))\,)-l_h(\boldsymbol{w}_h)|.
\]
Applying the triangle inequality and \eqref{norm_equal}, we have
\[
\|\boldsymbol{u}-\boldsymbol{u}_h\|_{*,h}\leq \|\boldsymbol{u}-\boldsymbol{v}_h\|_{*,h}+\|\boldsymbol{u}_h-\boldsymbol{v}_h\|_{*,h}\leq \interleave \boldsymbol{u}-\boldsymbol{v}_h\interleave_{*,h}+C_2\|\boldsymbol{w}_h\|_{1,h}.
\]
Combining the above results, we obtain the following version of Strang's second lemma:
\[
\begin{aligned}
\|\boldsymbol{u}-\boldsymbol{u}_h\|_{*,h}&\leq (1+C_AC_2^2C_a^{-1})\inf_{\boldsymbol{v}_h\in \boldsymbol{Z}_h}\left(\frac{}{}\interleave \boldsymbol{u}-\boldsymbol{v}_h\interleave_{*,h} \right.\\
&~~~~\left.\frac{}{}+\|R(\boldsymbol{u})-\pi_{h}^0 R(\boldsymbol{u})-R_h(\boldsymbol{v}_h)\|_{*,pre}\right)\\
&+\frac{1}{C_a}\sup_{\boldsymbol{w}_h\in \boldsymbol{Z}_h}\frac{|A_h(\,(\boldsymbol{u}, R(\boldsymbol{u})), (\boldsymbol{w}_h,R_h(\boldsymbol{w}_h))\,)-l_h(\boldsymbol{w}_h)|}{\|\boldsymbol{w}_h\|_{1,h}}
\end{aligned}
\]
Recalling \eqref{zh_pi}, we apply the interpolation error estimates \eqref{iter_err_new1} and \eqref{iter_err_new2}  to obtain
\[
\begin{aligned}
\inf_{\boldsymbol{v}_h\in \boldsymbol{Z}_h}&\left(\interleave \boldsymbol{u}-\boldsymbol{v}_h\interleave_{*,h}+\|R(\boldsymbol{u})-\pi_{h}^0 R(\boldsymbol{u})-R_h(\boldsymbol{v}_h)\|_{*,pre}\right)\\
&\leq C\interleave \boldsymbol{u}-\boldsymbol{\pi}_{h}^{\rm IFE} \boldsymbol{u}  \interleave_{*,h}+C\|R(\boldsymbol{u})-\pi_{h}^0 R(\boldsymbol{u})-R_h(\boldsymbol{\pi}_{h}^{\rm IFE} \boldsymbol{u} )\|_{*,pre}\\
&\leq  Ch\|\boldsymbol{u}\|_{H^2(\cup \Omega^\pm)}.
\end{aligned}
\]
Together with Lemma~\ref{lem_consis} and the relation \eqref{norm_rela}, this establishes the first error estimate \eqref{error_1}.

ii) Since $\pi_h^0 p-\hat{p}_h\in Y_h$, the discrete inf-sup stability from Lemma~\ref{lem_inf_sup} yields
\begin{equation}\label{pro_p_01}
\|\pi_h^0 p-\hat{p}_h\|_{L^2(\Omega)}\leq \frac{1}{\beta} \sup_{\boldsymbol{v}_h \in \boldsymbol{X}_h^{\rm IFE}} \frac{b_h(\boldsymbol{v}_h, \pi_h^0 p-\hat{p}_h)}{\| \boldsymbol{v}_h \|_{1,h}}.
\end{equation}

Next, we estimate the numerator of the right-hand side. Let $\hat{p}=p-R(\boldsymbol{u})$. 
For any  $\boldsymbol{v}_h \in \boldsymbol{X}_h^{\rm IFE}$, we have  $\nabla\cdot(\boldsymbol{\pi}_h^{\rm RT}\boldsymbol{v}_h)\in Y_h$.
Then, from the property of the $L^2$-projection and the commutativity property \eqref{commu_RT}, it follows that
\begin{equation}\label{pro_p_02}
\begin{aligned}
b_h(\boldsymbol{v}_h, \pi_h^0 p)&=-\sum_{T\in\mathcal{T}_h}\int_T \pi_h^0 p\nabla\cdot\boldsymbol{v}_h=-\sum_{T\in\mathcal{T}_h}\int_T (\pi_h^0 p) \, \pi_h^0(\nabla\cdot\boldsymbol{v}_h)\\
&=-\sum_{T\in\mathcal{T}_h}\int_T \pi_h^0 p  (\nabla \cdot \boldsymbol{\pi}_h^{\rm RT}\boldsymbol{v}_h)=-\sum_{T\in\mathcal{T}_h}\int_T  p  (\nabla \cdot \boldsymbol{\pi}_h^{\rm RT}\boldsymbol{v}_h)\\
&=-\sum_{T\in\mathcal{T}_h}\int_T  p  (\nabla \cdot \boldsymbol{\pi}_h^{\rm RT}\boldsymbol{v}_h)+\sum_{e\in\mathcal{E}_h^\Gamma}
\int_e\{p\}[\boldsymbol{\pi}_h^{\rm RT}\boldsymbol{v}_h]\cdot\boldsymbol{n}_e\\
&=b_h(\boldsymbol{\pi}_h^{\rm RT}\boldsymbol{v}_h,  p)=b_h(\boldsymbol{\pi}_h^{\rm RT}\boldsymbol{v}_h,  R(\boldsymbol{u}) )+b_h(\boldsymbol{\pi}_h^{\rm RT}\boldsymbol{v}_h,  \hat{p}),
\end{aligned}
\end{equation}
where we have used the fact that $[\boldsymbol{\pi}_h^{\rm RT}\boldsymbol{v}_h]\cdot\boldsymbol{n}_e=0$ on each edge $e\in\mathcal{E}_h^\Gamma$. 

On the other hand, using \eqref{IFE_method}, we obtain
\begin{equation}\label{pro_p_03}
-b_h(\boldsymbol{v}_h, \hat{p}_h)=A_h(\,(\boldsymbol{u}_h,R_h(\boldsymbol{u}_h)), (\boldsymbol{v}_h,R_h(\boldsymbol{v}_h))\, )-l_h(\boldsymbol{v}_h),
\end{equation}
where the two terms on the right-hand side  can be reformulated as 
\begin{equation}\label{pro_p_04}
\begin{aligned}
A_h(\,&(\boldsymbol{u}_h,R_h(\boldsymbol{u}_h)), (\boldsymbol{v}_h,R_h(\boldsymbol{v}_h))\, )\\
&=A_h(\,(\boldsymbol{u},R(\boldsymbol{u})), (\boldsymbol{v}_h,R_h(\boldsymbol{v}_h))\, )+A_h(\,(\boldsymbol{u}_h-\boldsymbol{u},R_h(\boldsymbol{u}_h)-R(\boldsymbol{u})), (\boldsymbol{v}_h,R_h(\boldsymbol{v}_h))\, )\\
&=A_h(\,(\boldsymbol{u},R(\boldsymbol{u})), (\boldsymbol{v}_h,R_h(\boldsymbol{v}_h))\, )-b_h(\boldsymbol{v}_h,\boldsymbol{\pi}_h^0R(\boldsymbol{u}))\\
&\quad+A_h(\,(\boldsymbol{u}_h-\boldsymbol{u},R_h(\boldsymbol{u}_h)+\boldsymbol{\pi}_h^0R(\boldsymbol{u})-R(\boldsymbol{u})), (\boldsymbol{v}_h,R_h(\boldsymbol{v}_h))\,)
\end{aligned}
\end{equation}
and
\begin{equation}\label{pro_p_05}
\begin{aligned}
-l_h(\boldsymbol{v}_h)&=-\int_\Omega \boldsymbol{f}\cdot \boldsymbol{\pi}_h^{\rm RT}\boldsymbol{v}_h=-\sum_{s=\pm}\int_{\Omega^s}(-\nabla\cdot (2\mu\boldsymbol{\epsilon}(\boldsymbol{u})) +\nabla p) \cdot \boldsymbol{\pi}_h^{\rm RT}\boldsymbol{v}_h\\
&=-\sum_{s=\pm}\int_{\Omega^s}(-\nabla\cdot (2\mu\boldsymbol{\epsilon}(\boldsymbol{u})) +\nabla R(\boldsymbol{u})+\nabla\hat{p}) \cdot \boldsymbol{\pi}_h^{\rm RT}\boldsymbol{v}_h\\
&=-\left(\sum_{s=\pm}\int_{\Omega^s}(-\nabla\cdot (2\mu\boldsymbol{\epsilon}(\boldsymbol{u})) +\nabla R(\boldsymbol{u})) \cdot \boldsymbol{\pi}_h^{\rm RT}\boldsymbol{v}_h\right)- b_h(\boldsymbol{\pi}_h^{\rm RT}\boldsymbol{v}_h, \hat{p}).
\end{aligned}
\end{equation}
Combining \eqref{pro_p_02}–\eqref{pro_p_05}, we have
\begin{equation}\label{pro_iii0}
\begin{aligned}
b_h(\boldsymbol{v}_h, &\pi_h^0 p-\hat{p}_h)=b_h(\boldsymbol{\pi}_h^{\rm RT}\boldsymbol{v}_h,  R(\boldsymbol{u}))-b_h(\boldsymbol{v}_h,\boldsymbol{\pi}_h^0R(\boldsymbol{u}))\\
&\quad+A_h(\,(\boldsymbol{u}_h-\boldsymbol{u},R_h(\boldsymbol{u}_h)+\boldsymbol{\pi}_h^0R(\boldsymbol{u})-R(\boldsymbol{u})), (\boldsymbol{v}_h,R_h(\boldsymbol{v}_h))\, )\\
&\quad+A_h(\,(\boldsymbol{u},R(\boldsymbol{u})), (\boldsymbol{v}_h,R_h(\boldsymbol{v}_h))\, )-\sum_{s=\pm}\int_{\Omega^s}(-\nabla\cdot (2\mu\boldsymbol{\epsilon}(\boldsymbol{u})) +\nabla R(\boldsymbol{u})) \cdot \boldsymbol{\pi}_h^{\rm RT}\boldsymbol{v}_h\\
&:=(\mathrm{III})_1+(\mathrm{III})_2+(\mathrm{III})_3.
\end{aligned}
\end{equation}
By Lemma~\ref{lem_cons_0}, the third term can be estimated as 
\begin{equation}\label{pro_iii1}
|(\mathrm{III})_3|\leq Ch\|\boldsymbol{u}\|_{H^2(\cup \Omega^\pm)} \|\boldsymbol{v}_h\|_{1,h}.
\end{equation}
To estimate the second term, we use \eqref{bd_Ah}, \eqref{norm_equal}, and \eqref{error_1} to obtain
\[
\begin{aligned}
|(\mathrm{III})_2|&\leq C\left(\interleave \boldsymbol{u}_h-\boldsymbol{u}  \interleave_{*,h}+\|R_h(\boldsymbol{u}_h)+\pi_{h}^0 R(\boldsymbol{u})-R(\boldsymbol{u})\|_{*,pre}\right)\|\boldsymbol{v}_h\|_{1,h}\\
&\leq C\left(h\| \boldsymbol{u}\|_{H^2(\cup \Omega^\pm)}+\|R_h(\boldsymbol{u}_h)+\pi_{h}^0 R(\boldsymbol{u})-R(\boldsymbol{u})\|_{*,pre}\right)\|\boldsymbol{v}_h\|_{1,h}.
\end{aligned}
\]
By \eqref{iter_err_new2}, \eqref{norm_equal}, \eqref{error_1} and \eqref{iter_err_new1}, we further have
\[
\begin{aligned}
\|R_h(\boldsymbol{u}_h)+\pi_{h}^0 R(\boldsymbol{u})-R(\boldsymbol{u})\|_{*,pre}
&\leq \|R_h(\boldsymbol{\pi}_h^{\rm IFE}\boldsymbol{u})+\pi_{h}^0 R(\boldsymbol{u})-R(\boldsymbol{u})\|_{*,pre}+\|R_h(\boldsymbol{u}_h-\boldsymbol{\pi}_h^{\rm IFE}\boldsymbol{u})\|_{*,pre}\\
&\leq Ch\| \boldsymbol{u}\|_{H^2(\cup \Omega^\pm)}+C_2\|\boldsymbol{u}_h-\boldsymbol{\pi}_h^{\rm IFE}\boldsymbol{u}\|_{1,h}\\
&\leq Ch\| \boldsymbol{u}\|_{H^2(\cup \Omega^\pm)}+C_2\|\boldsymbol{u}_h-\boldsymbol{u}\|_{1,h}+C_2\|\boldsymbol{u}-\boldsymbol{\pi}_h^{\rm IFE}\boldsymbol{u}\|_{1,h}\\
&\leq Ch\| \boldsymbol{u}\|_{H^2(\cup \Omega^\pm)}.
\end{aligned}
\]
Thus, the second term admits the estimate
\begin{equation}\label{pro_iii2}
|(\mathrm{III})_2|\leq Ch\|\boldsymbol{u}\|_{H^2(\cup \Omega^\pm)} \|\boldsymbol{v}_h\|_{1,h}.
\end{equation}
We now consider the first term $(\mathrm{III})_1$. Using the definition of the bilinear form $b_h(\cdot,\cdot)$ in \eqref{def_Ah}, and the facts that $[\boldsymbol{\pi}_h^{\rm RT}\boldsymbol{v}_h]\cdot\boldsymbol{n}_e = 0$ and $\{\boldsymbol{\pi}_h^0 R(\boldsymbol{u})\}$ is constant on each edge $e\in \mathcal{E}_h^\Gamma$, we obtain
\[
\begin{aligned}
(\mathrm{III})_1&=b_h(\boldsymbol{\pi}_h^{\rm RT}\boldsymbol{v}_h,  R(\boldsymbol{u}))-b_h(\boldsymbol{v}_h,\boldsymbol{\pi}_h^0R(\boldsymbol{u}))\\
&=-\sum_{T\in\mathcal{T}_h}\int_T R(\boldsymbol{u})\nabla\cdot\boldsymbol{\pi}_h^{\rm RT}\boldsymbol{v}_h+\sum_{T\in\mathcal{T}_h}\int_T (\boldsymbol{\pi}_h^0R(\boldsymbol{u}))\nabla\cdot\boldsymbol{v}_h.
\end{aligned}
\]
By the commutativity property \eqref{commu_RT},  the $L^2$-projection property, and the fact that $\nabla\cdot\boldsymbol{v}_h\in Y_h$, we have
\begin{equation}\label{pro_iii3}
\begin{aligned}
(\mathrm{III})_1&=-\sum_{T\in\mathcal{T}_h}\int_T R(\boldsymbol{u})\pi_h^0(\nabla\cdot \boldsymbol{v}_h)+\sum_{T\in\mathcal{T}_h}\int_T R(\boldsymbol{u})\nabla\cdot\boldsymbol{v}_h\\
&=-\sum_{T\in\mathcal{T}_h}\int_T R(\boldsymbol{u})\nabla\cdot \boldsymbol{v}_h+\sum_{T\in\mathcal{T}_h}\int_T R(\boldsymbol{u})\nabla\cdot\boldsymbol{v}_h\\
&=0.
\end{aligned}
\end{equation}
Combining \eqref{pro_iii0}-\eqref{pro_iii3}, we arrive at
\[
b_h(\boldsymbol{v}_h, \pi_h^0 p-\hat{p}_h)\leq Ch\|\boldsymbol{u}\|_{H^2(\cup \Omega^\pm)} \|\boldsymbol{v}_h\|_{1,h},
\]
which together with \eqref{pro_p_01} yields the desired estimate \eqref{error_2}.

iii)
Noting that $p_h=R_h(\boldsymbol{u}_h)+\hat{p}_h$, we apply the triangle inequality and \eqref{error_2} to obtain
\[
\begin{aligned}
\|p-p_h\|_{L^2(\Omega)}&=\|p-R_h(\boldsymbol{u}_h)-\pi_h^0p+\pi_h^0p-\hat{p}_h\|_{L^2(\Omega)}\\
&\leq \|p-R_h(\boldsymbol{u}_h)-\pi_h^0p\|_{L^2(\Omega)}+\|\pi_h^0p-\hat{p}_h\|_{L^2(\Omega)}\\
&\leq \|p-R_h(\boldsymbol{u}_h)-\pi_h^0p\|_{L^2(\Omega)}+Ch\|\boldsymbol{u}\|_{H^2(\cup \Omega^\pm)}.
\end{aligned}
\]
To estimate the first term on the right-hand side, we use \eqref{norm_equal}, \eqref{ultimate_ve1},  \eqref{ultimate_pr1} and \eqref{error_1} to obtain
\[
\begin{aligned}
\|p-R_h(\boldsymbol{u}_h)-\pi_h^0p\|_{L^2(\Omega)}&\leq \|p-R_h(\boldsymbol{\pi}_h^{\rm IFE}\boldsymbol{u})-\pi_h^0p\|_{L^2(\Omega)}+ \|R_h(\boldsymbol{\pi}_h^{\rm IFE}\boldsymbol{u}-\boldsymbol{u}_h)\|_{L^2(\Omega)}\\
&\leq \|p-R_h(\boldsymbol{\pi}_h^{\rm IFE}\boldsymbol{u})-\pi_h^0p\|_{L^2(\Omega)} +C_2 \|\boldsymbol{\pi}_h^{\rm IFE}\boldsymbol{u}-\boldsymbol{u}_h\|_{1,h}\\
&\leq \|p-R_h(\boldsymbol{\pi}_h^{\rm IFE}\boldsymbol{u})-\pi_h^0p\|_{L^2(\Omega)} +C_2 (\|\boldsymbol{\pi}_h^{\rm IFE}\boldsymbol{u}-\boldsymbol{u}\|_{1,h}+\|\boldsymbol{u}-\boldsymbol{u}_h\|_{1,h})\\
&\leq Ch(\|\boldsymbol{u}\|_{H^2(\cup \Omega^\pm)}+\|p\|_{H^1(\cup \Omega^\pm)}).
\end{aligned}
\]
Combining the above results completes the proof of the error estimate \eqref{error_3}.
\end{proof}

\section{Extension to Problems with Surface Tension}\label{sec_exten}

In this section, we consider the Stokes interface problem \eqref{originalpb0} in which the jump condition
\eqref{jp_cond1} is replaced by
\begin{equation}\label{nonhom_jump}
[\boldsymbol{\sigma}(\mu,\boldsymbol{u},p)\boldsymbol{n}]_\Gamma
= g\,\boldsymbol{n}
\qquad \text{on } \Gamma,
\end{equation}
where $g$ is a given function defined on the interface $\Gamma$.
 To model the surface tension according to Laplace’s law, the function $g$ is given by
$g = \gamma H_\Gamma$, where $\gamma$ is the surface tension coefficient and $H_\Gamma$ is the curvature of the interface~$\Gamma$ (see \cite{piccardo2023surface}). 
In this work, we consider a general function $g$ and assume that $g\in H^{1/2}(\Gamma)$. 
We note that the regularity property
\(H_\Gamma\in H^{1/2}(\Gamma)\) requires additional smoothness
assumptions on the interface.

\subsection{Correction Function at the Continuous Level}
From a theoretical point of view, the nonhomogeneous jump condition
can be treated in the same manner as a nonhomogeneous Dirichlet
boundary condition. That is, we construct a correction function at the continuous level that
satisfies the prescribed nonhomogeneous jump condition.
More precisely, given $g \in H^{1/2}(\Gamma)$, there exists
$w^- \in H^1(\Omega^-)$ such that
\[
w^- = g \quad \text{on } \Gamma\quad \text{ and }
\quad
\|w^-\|_{H^1(\Omega^-)}
\le C \|g\|_{H^{1/2}(\Gamma)}.
\]
Let $w^+ = 0$, and define $w$ by $w|_{\Omega^\pm} = w^\pm.$
We then define
\[
p^J(g) = w - \frac{1}{|\Omega|}\int_{\Omega} w \, dx,
\]
and, for brevity, write $p^J:=p^J(g)$ when no confusion arises.  Therefore,
\[
p^J \in L_0^2(\Omega)\quad \text{ and } \quad [p^J]_\Gamma = -g.
\]
Moreover, we have
\begin{equation}\label{esti_pJ}
\|p^J\|_{H^1(\cup \Omega^\pm)} \le C \|g\|_{H^{1/2}(\Gamma)}.
\end{equation}
By construction,
\[
[\boldsymbol{\sigma}(\mu,\boldsymbol{0},p^J)\boldsymbol{n}]_\Gamma=-[p^J]_\Gamma \,\boldsymbol{n}=g\,\boldsymbol{n}.
\]
Let $p^\star = p - p^J$. Then $p^\star \in L_0^2(\Omega)$  and the jump condition
\eqref{nonhom_jump} is transformed into 
\begin{equation}\label{jump_trans}
[\boldsymbol{\sigma}(\mu,\boldsymbol{u},p^\star)\boldsymbol{n}]_\Gamma
=
\boldsymbol{0}.
\end{equation}
Accordingly, equation \eqref{originalpb1} becomes
\[
-\nabla\cdot \boldsymbol{\sigma}(\mu,\boldsymbol{u},p^\star)
=
\boldsymbol{f}
-
\nabla p^J
\qquad \text{in } \Omega^+ \cup \Omega^- .
\]
Therefore, the proposed method can be applied to this modified problem.

\subsection{Discrete Correction Function}
While the above construction is convenient for theoretical analysis,
it is not employed in practical computations, since constructing
$p^J$ requires solving an additional boundary-value problem. 
In the discrete scheme, we use a discrete correction function $p_h^J$, supported only on interface elements, to enforce the jump condition at the discrete level, without solving an additional boundary-value problem. 
Next, we describe the construction of \( p_h^J \) proposed in \cite{zhang2023unfitted}.  
For each interface element \( T \in \mathcal{T}_h^{\Gamma} \), let \( \Gamma_T^{\mathrm{ext}} \) be a portion of \( \Gamma \) containing \( \Gamma_T \) such that \(C_1h  \leq |\Gamma_T^{\mathrm{ext}}| \leq C_2 h \), where $C_1$ and $C_2$ are constants  independent of $h$ and the position of the interface relative to the mesh. 
Given an interface element \(T\in\mathcal{T}_h^\Gamma\), define
\[
w^-=\frac{1}{|\Gamma_T^{\mathrm{ext}}|}
\int_{\Gamma_T^{\mathrm{ext}}} g\,ds,
\qquad
w^+=0,
\]
and let \(w\) be the piecewise constant function defined by
\(
w|_{T_h^\pm}=w^\pm .
\)
We then define
\[
p_h^J(g)|_T
=
w-\frac{1}{|T|}\int_T w\,dx .
\]
On each non-interface element \( T \in \mathcal{T}_h^{non} \), we set \( p_h^J(g)|_{T} = 0 \). Also, for brevity, we write $p_h^J:=p_h^J(g)$ when no confusion arises.
It is straightforward to verify that \( p_h^J \in L_0^2(\Omega) \) and that $p_h^J$ satisfies a discrete interface jump condition approximating that of \( p^J \).  Combining  Theorem 4.8 in \cite{zhang2023unfitted} with \eqref{esti_pJ}, we obtain the following approximation estimate:
\begin{equation}\label{appro_p_h^J}
\|p^J-(\pi_h^0p^J+p_h^J)\|_{L^2(\Omega)}\leq Ch\|p^J\|_{H^1(\cup \Omega^\pm)}\leq Ch \|g\|_{H^{1/2}(\Gamma)}.
\end{equation}
By a slight modification of the proof of Theorem 4.8 in \cite{zhang2023unfitted}, we also obtain
\begin{equation}\label{appro_p_h^J_pm}
\sum_{T\in\mathcal{T}_h^\Gamma}\|(p^J)_E^\pm-(\pi_h^0p^J+p_h^J)^\pm \|^2_{L^2(T)} \leq Ch^2 \|g\|^2_{H^{1/2}(\Gamma)}.
\end{equation}

\subsection{The Pressure-Robust Method}

By \eqref{jump_trans}, we know that $(\boldsymbol{u},p^\star) \in \widetilde{\boldsymbol{H}^2H^1}\cap (\boldsymbol{V}\times Q)$. Based on the space decomposition \eqref{decomp_con}, we then have
\begin{equation}\label{deco_p}
p=p^\star+p^J=\hat{p}+R(\boldsymbol{u})+p^J \quad \text{ with }\quad \hat{p}\in Y.
\end{equation}
Similar to \eqref{weakform1} for the case \( g = 0 \), the exact solutions \( \boldsymbol{u} \in \boldsymbol{X} \) and
\(
p = \hat{p} + R(\boldsymbol{u}) + p^J \text{ with } \hat{p} \in Y
\) for the case \( g \neq 0 \) satisfy the following relations:
\begin{equation}\label{weak_surface}
\begin{aligned}
A\big((\boldsymbol{u}, R(\boldsymbol{u})), (\boldsymbol{v}, R(\boldsymbol{v}))\big) + b(\boldsymbol{v}, \hat{p})
&= l(\boldsymbol{v})+l^g(\boldsymbol{v})-b(\boldsymbol{v}, p^J) \quad &&\forall \, \boldsymbol{v} \in \boldsymbol{X},\\
b(\boldsymbol{u}, \hat{q})
&= 0 \quad &&\forall \, \hat{q} \in Y,
\end{aligned}
\end{equation}
where $l^g(\boldsymbol{v})=-\int_\Gamma g\,\boldsymbol{n}\cdot \boldsymbol{v}$.

Accordingly,  the pressure-robust IFE method  for \( g \neq 0 \) reads: find $(\boldsymbol{u}_h, p_h)$, with $p_h=R_h(\boldsymbol{u}_h)+\hat{p}_h+p_h^J$ and $(\boldsymbol{u}_h, \hat{p}_h)\in  \boldsymbol{X}_h^{\rm IFE}\times Y_h$, such that 
\begin{equation}\label{IFE_method_surface}
\begin{aligned}
A_h(\,(\boldsymbol{u}_h,R_h(\boldsymbol{u}_h)), (\boldsymbol{v}_h,R_h(\boldsymbol{v}_h))\, )+ b_h(\boldsymbol{v}_h, \hat{p}_h)&=l_h^{\rm tot}(\boldsymbol{v}_h) \quad &&\forall \,  \boldsymbol{v}_h\in \boldsymbol{X}_h^{\rm IFE},\\
b_h(\boldsymbol{u}_h,\hat{q}_h)&=0\quad &&\forall\,   \hat{q}_h\in Y_h,
\end{aligned}
\end{equation}
where the total load functional on the right-hand side is defined by
\begin{equation}\label{l_tot}
l_h^{\rm tot}(\boldsymbol{v}_h)=l_h(\boldsymbol{v}_h)+l^g(\boldsymbol{v}_h)-b_h(\boldsymbol{v}_h, p_h^J).
\end{equation}

Note that if $g=0$, then $p_h^J=p^J=0$, and both the above weak formulation \eqref{weak_surface} and the discrete scheme \eqref{IFE_method_surface} reduce to the case without surface tension.

\subsection{Error Estimates}
We first establish the following consistency result, analogous to Lemma~\ref{lem_consis}.
\begin{lemma}
Assume that the exact velocity of problem~(\ref{weak_surface}) satisfies $\boldsymbol{u}\in H^2(\cup \Omega^\pm)^2$. For the discrete scheme~\eqref{IFE_method_surface}, there exists a constant \( C \), independent of \( h \) and the position of the interface relative to the mesh, such that
\[
\sup_{\boldsymbol{w}_h \in \boldsymbol{Z}_h}
\frac{
\left|
A_h\big((\boldsymbol{u}, R(\boldsymbol{u})), (\boldsymbol{w}_h, R_h(\boldsymbol{w}_h))\big)
- l_h^{\rm tot}(\boldsymbol{w}_h)
\right|
}{
\|\boldsymbol{w}_h\|_{1,h}
}
\le C h ( \|\boldsymbol{u}\|_{H^2(\cup\Omega^\pm)} + \|g\|_{H^{1/2}(\Gamma)} ).
\]
\end{lemma}

\begin{proof}
Similar to \eqref{pro_lh_er}, using \eqref{deco_p}, we obtain
\[
l_h(\boldsymbol{w}_h)=\int_\Omega \boldsymbol{f}\cdot \boldsymbol{\pi}_h^{\rm RT}\boldsymbol{w}_h=\sum_{s=\pm}\int_{\Omega^s}-\nabla\cdot \boldsymbol{\sigma}(\mu,\boldsymbol{u}, R(\boldsymbol{u})+ p^J)\cdot (\boldsymbol{w}_h+(\boldsymbol{\pi}_h^{\rm RT}\boldsymbol{w}_h-\boldsymbol{w}_h)).
\]
Then, by \eqref{l_tot} and the triangle inequality, we have
\begin{equation}\label{pro_decom_IV}
\begin{aligned}
&\left|A_h(\,(\boldsymbol{u}, R(\boldsymbol{u})), (\boldsymbol{w}_h,R_h(\boldsymbol{w}_h))\,)-l_h^{\rm tot}(\boldsymbol{w}_h))\right|
\\
&\quad \leq \left| A_h(\cdot,\cdot)- \left(\sum_{s=\pm}\int_{\Omega^s}-\nabla\cdot \boldsymbol{\sigma}(\mu,\boldsymbol{u}, R(\boldsymbol{u})+ p^J)\cdot \boldsymbol{w}_h+l^g(\boldsymbol{w}_h)-b_h(\boldsymbol{w}_h, p^J) \right)\right|\\
&\qquad + \left| \sum_{s=\pm}\int_{\Omega^s}-\nabla\cdot \boldsymbol{\sigma}(\mu,\boldsymbol{u}, R(\boldsymbol{u})+ p^J)\cdot (\boldsymbol{\pi}_h^{\rm RT}\boldsymbol{w}_h-\boldsymbol{w}_h) \right|\\
&\qquad +\left| b_h(\boldsymbol{w}_h, p^J)-b_h(\boldsymbol{w}_h, p_h^J)\right|\\
&\quad := (\mathrm{IV})_1 + (\mathrm{IV})_2+ (\mathrm{IV})_3,
\end{aligned}
\end{equation}
where $A_h(\cdot,\cdot):=A_h(\,(\boldsymbol{u}, R(\boldsymbol{u})), (\boldsymbol{w}_h,R_h(\boldsymbol{w}_h))\,).$
Integration by parts yields
\[
 \sum_{s=\pm}\int_{\Omega^s}-\nabla\cdot \boldsymbol{\sigma}(\mu,\boldsymbol{0},  p^J)\cdot \boldsymbol{w}_h+l^g(\boldsymbol{w}_h)-b_h(\boldsymbol{w}_h, p^J)=\sum_{e\in\mathcal{E}_h^{non}}\int_e\{p^J\}[\boldsymbol{w}_h]\cdot\boldsymbol{n}_e.
\]
Recalling the term \((\mathrm{I})_1\) in \eqref{pro_deco_Ah} and using the fact that \(\int_e [\boldsymbol{w}_h]=\boldsymbol{0}\), we have
\[
\begin{aligned}
 (\mathrm{IV})_1&\leq (\mathrm{I})_1 + \sum_{e\in\mathcal{E}_h^{non}}\int_e\left| \{p^J\}[\boldsymbol{w}_h]\cdot\boldsymbol{n}_e \right| \\
 &\leq Ch \|\boldsymbol{u}\|_{H^2(\cup\Omega^\pm)} \|\boldsymbol{w}_h\|_{1,h}+Ch\|p^J\|_{H^1(\cup \Omega^\pm)}\|\boldsymbol{w}_h\|_{1,h}\\
 &\leq Ch (\|\boldsymbol{u}\|_{H^2(\cup\Omega^\pm)}+\|g\|_{H^{1/2}(\Gamma)}) \|\boldsymbol{w}_h\|_{1,h}.
 \end{aligned}
\]
Following the same argument as in the estimate of term $(\mathrm{I})_2$ in \eqref{pro_deco_Ah}, we obtain
\[
 (\mathrm{IV})_2\leq Ch( \|\boldsymbol{u}\|_{H^2(\cup\Omega^\pm)} + \|g\|_{H^{1/2}(\Gamma)} )\|\boldsymbol{w}_h\|_{1,h}.
\]
For the third term, similar to \eqref{pro_bh_pi}, we have $b_h(\boldsymbol{w}_h, \pi_h^0p^J)=0.$
Therefore, 
\[
\begin{aligned}
(\mathrm{IV})_3&=\left| b_h(\boldsymbol{w}_h, p^J-(\pi_h^0p^J+p_h^J))\right|\leq \sum_{e\in\mathcal{E}_h^\Gamma}\left| \int_e \{ p^J-(\pi_h^0p^J+p_h^J) \} [\boldsymbol{w}_h]\cdot \boldsymbol{n}_e\right|,
\end{aligned}
\]
where we used  \eqref{w_h_eq_0} since $\boldsymbol{w}_h \in \boldsymbol{Z}_h$.
By \eqref{appro_p_h^J_pm} and the same argument as in Lemma~\ref{ener_app}, we have
\[
\sum_{e\in\mathcal{E}_h^\Gamma}\|\{ p^J-(\pi_h^0p^J+p_h^J) \}\|^2_{L^2(e)}\leq Ch \|g\|^2_{H^{1/2}(\Gamma)}.
\]
Moreover, since \(\boldsymbol{w}_h|_T \in H^1(T)\), we have (cf. (4.6) in \cite{ji2023analysis}) 
\[
\sum_{e\in\mathcal{E}_h^\Gamma} \|[\boldsymbol{w}_h]\|^2_{L^2(e)}\leq Ch\|\boldsymbol{w}_h\|^2_{1,h}.
\]
Combining the above results and applying the Cauchy--Schwarz inequality, we conclude that
\[
(\mathrm{IV})_3\leq Ch \|g\|_{H^{1/2}(\Gamma)}\|\boldsymbol{w}_h\|_{1,h}.
\]
Substituting the estimates of \((\mathrm{IV})_1\)--\((\mathrm{IV})_3\) into \eqref{pro_decom_IV}, we arrive at the desired estimate.
\end{proof}

With the above consistency result, we can prove the following error estimates, extending Theorem~\ref{theo_ulti} to the case $g\neq 0$. The proof follows the same lines as that of Theorem~\ref{theo_ulti} and is therefore omitted.
\begin{theorem}\label{theo_ulti_nonhom}
Let $(\boldsymbol{u}, p)$ be the exact solution of  problem \eqref{originalpb0} with 
\eqref{jp_cond1} replaced by \eqref{nonhom_jump}.  Assume that $\boldsymbol{u}\in H^2(\cup \Omega^\pm)^2$ and $p\in H^1(\cup \Omega^\pm)$.
Let $(\boldsymbol{u}_h, p_h)$ be the solution of  \eqref{IFE_method_surface} with $p_h=R_h(\boldsymbol{u}_h)+\hat{p}_h+p_h^J$. We have
\begin{align}
\|\boldsymbol{u}-\boldsymbol{u}_h\|_{1,h}&\leq C\|\boldsymbol{u}-\boldsymbol{u}_h\|_{*,h}\leq Ch( \|\boldsymbol{u}\|_{H^2(\cup\Omega^\pm)} + \|g\|_{H^{1/2}(\Gamma)} ),\label{nohom_error_1}\\
\|\pi_h^0p-\hat{p}_h\|_{L^2(\Omega)}&=\|\pi_h^0 (p- p_h^J)-\hat{p}_h\|_{L^2(\Omega)}\leq Ch( \|\boldsymbol{u}\|_{H^2(\cup\Omega^\pm)} + \|g\|_{H^{1/2}(\Gamma)} ),\label{nohom_error_2}\\
\|p-p_h\|_{L^2(\Omega)}&\leq Ch(\|\boldsymbol{u}\|_{H^2(\cup \Omega^\pm)}+\|p\|_{H^1(\cup \Omega^\pm)}+ \|g\|_{H^{1/2}(\Gamma)} ),\label{nohom_error_3}
\end{align}
where  the constant $C > 0$ is independent of $h$ and the position of the interface relative to the mesh.
\end{theorem}

\section{Numerical Experiments}\label{sec_num}
In this section, we present several numerical experiments to validate the theoretical analysis. The computational domain is taken as $\Omega = (-1,1)\times(-1,1)$, and uniform triangulations are constructed by dividing the domain into $N \times N$ congruent rectangles. Each rectangle is then split along a fixed diagonal direction to form a conforming triangulation of triangles.
We compute the following relative errors and their corresponding convergence rates on a sequence of uniform triangulations:
  $e_0(\boldsymbol{u})=\frac{\|\boldsymbol{u}-\boldsymbol{u}_h\|_{L^2(\Omega)}}{\|\boldsymbol{u}\|_{L^2(\Omega)}},~ e_1(\boldsymbol{u})=\frac{\|\boldsymbol{u}-\boldsymbol{u}_h\|_{1,h}}{\|\boldsymbol{u}\|_{1,h}}, ~e_0(p)=\frac{\|p-p_h\|_{L^2(\Omega)}}{\|p\|_{L^2(\Omega)}}.$
 For comparison, we also compute the errors produced by the classical CR-IFE method, which corresponds to method \eqref{IFE_method} with the right-hand side $l_h(\boldsymbol{v}_h)$ replaced by
$
\int_\Omega \boldsymbol{f}\cdot \boldsymbol{v}_h.
$
All examples are tested with $\mu_{\min} := \min\{\mu^+, \mu^-\} = 1$. For a general value of $\mu_{\min}$, equations~\eqref{originalpb1} and~\eqref{jp_cond1} can be rescaled as
\[-\nabla\cdot (2 \mu_{\rm new}\boldsymbol{\epsilon}(\boldsymbol{u})) +\nabla p_{\rm new}=\boldsymbol{f}_{\rm new}\quad\mbox{ and }\quad [\boldsymbol{\sigma}(\mu_{\rm new},\boldsymbol{u},p_{\rm new})\boldsymbol{n}]_\Gamma=g_{\rm new}\boldsymbol{n},\]
where $\mu_{\rm new}=\mu/\mu_{\rm min}$, $p_{\rm new}=p/\mu_{\rm min}$, $g_{\rm new}=g/\mu_{\rm min}$, and $\boldsymbol{f}_{\rm new}=\boldsymbol{f}/\mu_{\rm min}$.

\textbf{Example 1.}
The interface is defined as $\Gamma = \{(x_1,x_2)^\top \in \mathbb{R}^2 : x_1^2 + x_2^2 = r_0^2\}$ with $r_0 = 0.5$. The exact solution $(\boldsymbol{u}, p)$ is prescribed for all $\boldsymbol{x} = (x_1, x_2)^\top$ by
\begin{equation*}
\boldsymbol{u}(\boldsymbol{x})=\left\{
\begin{aligned}
&\frac{r_0^2-|\boldsymbol{x}|^2}{\mu^-}\left(
\begin{array}{c}
-x_2\\
x_1
\end{array}
\right)~~~\mbox{ if }~ |\boldsymbol{x}|<r_0,\\
&\frac{r_0^2-|\boldsymbol{x}|^2}{\mu^+}\left(
\begin{array}{c}
-x_2\\
x_1
\end{array}
\right)~~~\mbox{ if }~  |\boldsymbol{x}|\geq r_0,\\
\end{aligned}\right.
~~\mbox{and }~~p(\boldsymbol{x})=(x_2^2-x_1^2)p_0.
\end{equation*}
The right-hand side $\boldsymbol{f}$ and the boundary condition $\boldsymbol{u}|_{\partial \Omega}$ are derived from the exact solution. It is easy to verify that  the exact solution satisfies the homogeneous interface conditions.
We set $\theta = -1$ and $\eta = 0$, and adopt a standard finite element approach to handle the nonhomogeneous Dirichlet boundary condition.
We test the example under various viscosity contrasts,  ranging from  small  to  large jumps: $\mu^+=5$, $\mu^-=1$; $\mu^+=1000$, $\mu^-=1$; $\mu^+=1$, $\mu^-=1000$, and with different pressure parameters: $p_0=1$, $p_0=10^6$.
The errors and convergence rates are reported in Tables~\ref{ex_biao1}–\ref{ex_biao3}. It can be observed from these tables that both the classical CR-IFE method and the pressure-robust CR-IFE method achieve optimal convergence rates. 
Moreover, the velocity errors obtained by the pressure-robust CR-IFE method remain unchanged when increasing $p_0$ from $1$ to $10^6$, thereby confirming the theoretical results in Theorem~\ref{theo_ulti}. We also test the $L^2$-norm of the elementwise divergence of the discrete velocity and find that it is zero up to machine precision.
In contrast, for the classical CR-IFE method, the velocity errors increase substantially as $p_0$ increases from $1$ to $10^6$.

\begin{table}[H]
\caption{Numerical results for Example 1 with  $\mu^+=5$, $\mu^-=1$.\label{ex_biao1}}
\begin{center}
{\small
\begin{tabular}{cc cc cc c}
  \hline
       $N$  &  $e_0(\boldsymbol{u})$   &  rate   &  $e_1(\boldsymbol{u})$  &  rate  &  $e_0(p)$  &  rate  \\ \hline
        \multicolumn{7}{c}{By the classical CR-IFE method,~$p_0=1$}  \\ \hline
          32  &    2.927E-03  &              &   5.051E-02 &              &   6.445E-02   &         \\ \hline
        64    &  7.313E-04    &  2.00    &  2.544E-02    &  0.99    &  3.206E-02     & 1.01\\ \hline
       128   &   1.875E-04   &   1.96   &   1.278E-02   &   0.99   &   1.599E-02    &  1.00\\ \hline
       256    &  4.696E-05   &   2.00   &   6.400E-03   &   1.00  &    7.983E-03    &  1.00\\ \hline
               \multicolumn{7}{c}{By the classical CR-IFE method,~$p_0=10^6$}  \\ \hline
        32 &     7.109E+02   &            &   7.431E+03  &                &   5.718E-02   &            \\ \hline
        64  &    1.779E+02  &    2.00  &    3.701E+03  &    1.01   &   2.857E-02   &   1.00 \\ \hline
       128  &    4.444E+01 &     2.00 &     1.845E+03  &    1.00   &   1.428E-02  &    1.00 \\ \hline
       256  &    1.110E+01 &     2.00  &    9.209E+02  &    1.00   &   7.141E-03   &   1.00 \\ \hline
                 \multicolumn{7}{c}{By the pressure-robust CR-IFE method,~$p_0=1$}  \\ \hline
        32   &   3.470E-03    &            &   6.022E-02   &             &  6.714E-02  &         \\ \hline
        64   &   8.692E-04    &  2.00   &   3.025E-02   &   0.99    &  3.335E-02  &    1.01\\ \hline
       128  &    2.200E-04   &   1.98  &    1.516E-02   &   1.00   &   1.662E-02  &    1.00\\ \hline
       256  &    5.501E-05   &   2.00  &    7.587E-03   &   1.00  &    8.300E-03   &   1.00\\ \hline
                        \multicolumn{7}{c}{By the pressure-robust CR-IFE method,~$p_0=10^6$}  \\ \hline
         32 &     3.470E-03  &             &   6.022E-02   &             &  5.706E-02   &        \\ \hline
        64  &    8.692E-04   &   2.00   &   3.025E-02   &   0.99   &   2.853E-02   &   1.00 \\ \hline
       128 &     2.200E-04  &    1.98  &    1.516E-02  &    1.00  &    1.426E-02  &    1.00 \\ \hline
       256 &     5.501E-05   &   2.00  &    7.587E-03   &   1.00  &    7.132E-03   &   1.00 \\ \hline
\end{tabular}
}
\end{center}
\end{table}

\begin{table}[H]
\caption{Numerical results for Example 1 with  $\mu^+=1000$, $\mu^-=1$.\label{ex_biao2}}
\begin{center}
{\small
\begin{tabular}{cc cc cc c}
  \hline
       $N$  &  $e_0(\boldsymbol{u})$   &  rate   &  $e_1(\boldsymbol{u})$  &  rate  &  $e_0(p)$  &  rate  \\ \hline
        \multicolumn{7}{c}{By the classical CR-IFE method,~$p_0=1$}  \\ \hline
        32 &     5.353E-01  &              &  5.426E-01   &            &    2.290E+00   &   \\ \hline
        64  &    2.215E-01  &    1.27   &   2.438E-01    &  1.15  &    1.919E+00   &   0.26\\ \hline
       128  &    6.979E-02  &    1.67  &    9.777E-02   &   1.32 &     1.176E+00  &    0.71\\ \hline
       256  &    2.056E-02  &    1.76  &    4.076E-02   &   1.26  &    6.442E-01   &   0.87\\ \hline
               \multicolumn{7}{c}{By the classical CR-IFE method,~$p_0=10^6$}  \\ \hline
        32  &    1.585E+02 &                &  2.232E+02  &              &    5.725E-02 &      \\ \hline
        64   &   6.919E+01  &    1.20    &  1.050E+02  &    1.09   &   2.859E-02  &    1.00\\ \hline
       128  &    2.258E+01  &    1.62   &   4.720E+01 &     1.15  &    1.429E-02 &     1.00\\ \hline
       256  &    6.722E+00   &   1.75   &   2.163E+01  &    1.13  &    7.143E-03 &     1.00\\ \hline
                 \multicolumn{7}{c}{By the pressure-robust CR-IFE method,~$p_0=1$}  \\ \hline
       32   &   5.347E-01    &             &  5.419E-01     &           &   2.287E+00    &         \\ \hline
        64  &    2.213E-01   &   1.27   &   2.435E-01    &  1.15  &    1.916E+00    &  0.26\\ \hline
       128 &     6.972E-02  &    1.67  &    9.767E-02   &   1.32  &    1.175E+00  &    0.71\\ \hline
       256  &    2.054E-02  &    1.76   &   4.073E-02   &   1.26  &    6.435E-01  &    0.87\\ \hline
                        \multicolumn{7}{c}{By the pressure-robust CR-IFE method,~$p_0=10^6$}  \\ \hline
         32  &    5.347E-01  &             &   5.419E-01  &             &   5.706E-02  &          \\ \hline
        64  &    2.213E-01   &   1.27   &   2.435E-01   &   1.15   &   2.853E-02  &    1.00\\ \hline
       128 &     6.972E-02  &    1.67  &    9.767E-02  &    1.32   &   1.426E-02 &     1.00\\ \hline
       256 &     2.054E-02   &   1.76  &    4.073E-02   &   1.26   &   7.132E-03  &    1.00\\ \hline
\end{tabular}
}
\end{center}
\end{table}

\begin{table}[H]
\caption{Numerical results for Example 1 with  $\mu^+=1$, $\mu^-=1000$.\label{ex_biao3}}
\begin{center}
{\small
\begin{tabular}{cc cc cc c}
  \hline
       $N$  &  $e_0(\boldsymbol{u})$   &  rate   &  $e_1(\boldsymbol{u})$  &  rate  &  $e_0(p)$  &  rate  \\ \hline
        \multicolumn{7}{c}{By the classical CR-IFE method,~$p_0=1$}  \\ \hline
        32  &    2.300E-02  &              &   5.583E-02  &              &   2.805E+01  &      \\ \hline      
        64  &    7.559E-03  &    1.61   &   2.598E-02  &    1.10   &   1.438E+01  &    0.96\\ \hline
       128 &     2.205E-03  &    1.78   &   1.209E-02 &     1.10  &    7.286E+00 &     0.98\\ \hline
       256 &     6.271E-04  &    1.81   &   5.722E-03 &     1.08  &    3.669E+00  &    0.99\\ \hline
               \multicolumn{7}{c}{By the classical CR-IFE method,~$p_0=10^6$}  \\ \hline
        32   &   6.903E+00   &            &    2.208E+01  &             &    5.800E-02  &           \\ \hline
        64  &    2.462E+00  &    1.49  &    1.037E+01  &    1.09  &    2.898E-02  &    1.00\\ \hline
       128 &     7.305E-01  &    1.75   &   4.907E+00  &    1.08  &    1.448E-02  &    1.00\\ \hline
       256  &    2.098E-01  &    1.80   &   2.358E+00   &   1.06  &    7.239E-03   &   1.00\\ \hline
                 \multicolumn{7}{c}{By the pressure-robust CR-IFE method,~$p_0=1$}  \\ \hline
        32  &    2.297E-02   &             &   5.579E-02 &              &    2.803E+01 &            \\ \hline
        64  &    7.551E-03   &   1.61   &   2.596E-02  &    1.10   &   1.436E+01  &    0.96\\ \hline
       128  &    2.203E-03  &    1.78  &    1.208E-02  &    1.10  &    7.279E+00 &     0.98\\ \hline
       256  &    6.265E-04  &    1.81  &    5.718E-03  &    1.08  &    3.666E+00  &    0.99\\ \hline
                        \multicolumn{7}{c}{By the pressure-robust CR-IFE method,~$p_0=10^6$}  \\ \hline
        32 &     2.297E-02   &           &     5.579E-02   &            &     5.706E-02 &             \\ \hline
        64  &    7.551E-03   &   1.61 &     2.596E-02   &   1.10  &    2.853E-02  &    1.00\\ \hline
       128  &    2.203E-03  &    1.78 &     1.208E-02  &    1.10  &    1.426E-02  &    1.00\\ \hline
       256  &    6.265E-04  &    1.81  &    5.718E-03   &   1.08  &    7.132E-03  &    1.00\\ \hline
\end{tabular}
}
\end{center}
\end{table}

\textbf{Example 2.} 
This example is designed to verify that the proposed pressure-robust CR-IFE method preserves the invariance property: modifying the body force by adding a gradient field affects only the pressure, not the velocity. We consider Example 1 with the force and pressure modified as $\boldsymbol{f} \rightarrow \boldsymbol{f} + \nabla \psi$ and $p \rightarrow p + \psi$, where $\psi = 10^6 x_1 x_2$.
The parameters are set as $\mu^+ = 5$, $\mu^- = 1$, and $p_0 = 1$. 
The numerical results in Table~\ref{ex_biao21} clearly demonstrate that the pressure-robust CR-IFE method yields substantially lower velocity errors than the classical CR-IFE method. 
Moreover, compared with Table~\ref{ex_biao1}, the velocity errors of the pressure-robust CR-IFE method remain unchanged despite the modification of the body force. This confirms the invariance property of the pressure-robust CR-IFE method.

\begin{table}[H]
\caption{Numerical results for Example 2. \label{ex_biao21}}
\begin{center}
{\small
\begin{tabular}{cc cc cc c}
  \hline
       $N$  &  $e_0(\boldsymbol{u})$   &  rate   &  $e_1(\boldsymbol{u})$  &  rate  &  $e_0(p)$  &  rate  \\ \hline
        \multicolumn{7}{c}{By the classical CR-IFE method}  \\ \hline
         32  &    4.158E+02  &             &   3.656E+03   &           &    3.707E-02   &           \\ \hline
        64   &   1.055E+02   &   1.98   &   1.839E+03    &  0.99  &    1.824E-02   &   1.02  \\ \hline
       128  &    2.649E+01  &    1.99  &    9.203E+02   &   1.00   &   9.061E-03  &    1.01  \\ \hline
       256  &    6.630E+00   &   2.00    &  4.600E+02   &   1.00     & 4.521E-03   &   1.00\\ \hline
                        \multicolumn{7}{c}{By the pressure-robust CR-IFE method }  \\ \hline
        32 &     3.470E-03  &             &    6.022E-02  &             &    3.607E-02   &          \\ \hline
        64 &     8.692E-04   &   2.00  &    3.025E-02  &    0.99   &   1.804E-02    &  1.00 \\ \hline
       128 &     2.200E-04  &    1.98  &    1.516E-02  &    1.00  &    9.021E-03  &    1.00 \\ \hline
       256 &     5.501E-05  &    2.00  &    7.587E-03  &    1.00  &    4.511E-03   &   1.00 \\ \hline
\end{tabular}
}
\end{center}
\end{table}

\textbf{Example 3.} 
Finally, we test the proposed method for the interface problem with surface tension. 
The exact velocity and the interface are the same as those in Example 1. 
The exact pressure is chosen as
\[
p(x_1,x_2) = (x_2^2 - x_1^2)p_0 + \psi(x_1,x_2) ,
\]
where $\psi|_{\Omega^+} = x_1 x_2$ and $\psi|_{\Omega^-} = \sin(x_1)\cos(x_2)$. Obviously, $p\in L_0^2(\Omega)$.
The viscosities are set as $\mu^+ = 5$ and $\mu^- = 1$. 
In this case, the interface condition is nonhomogeneous, i.e., $g=-[\psi]_\Gamma \neq 0$.

The numerical results are reported in Table~\ref{ex3_biao} and show that both the classical CR-IFE method and the pressure-robust CR-IFE method achieve optimal convergence rates, and that the velocity errors of the pressure-robust CR-IFE method remain essentially unchanged as $p_0$ increases from $1$ to $10^6$.

\begin{table}[H]
\caption{Numerical results for Example 3.\label{ex3_biao}}
\begin{center}
{\small
\begin{tabular}{cc cc cc c}
  \hline
       $N$  &  $e_0(\boldsymbol{u})$   &  rate   &  $e_1(\boldsymbol{u})$  &  rate  &  $e_0(p)$  &  rate  \\ \hline
        \multicolumn{7}{c}{By the classical CR-IFE method,~$p_0=1$}  \\ \hline
        32  &    2.965E-03   &            &     5.072E-02  &             &    5.551E-02   &           \\ \hline
        64   &   7.409E-04   &   2.00   &   2.554E-02   &   0.99   &   2.757E-02    &  1.01 \\ \hline
       128  &    1.899E-04  &    1.96  &    1.282E-02  &    0.99  &    1.375E-02   &   1.00 \\ \hline
       256  &    4.755E-05  &    2.00  &    6.425E-03  &    1.00  &    6.863E-03   &   1.00 \\ \hline
               \multicolumn{7}{c}{By the classical CR-IFE method,~$p_0=10^6$}  \\ \hline
        32  &    7.109E+02  &             & 7.431E+03     &             & 5.718E-02      &            \\ \hline
        64  &    1.779E+02  &    2.00  &    3.701E+03  &    1.01  &    2.857E-02   &   1.00  \\ \hline
       128 &     4.444E+01 &     2.00  &    1.845E+03  &    1.00 &     1.428E-02  &    1.00  \\ \hline
       256 &     1.110E+01  &    2.00   &   9.209E+02   &   1.00  &    7.141E-03   &   1.00  \\ \hline
                 \multicolumn{7}{c}{By the pressure-robust CR-IFE method,~$p_0=1$}  \\ \hline
        32  &    3.478E-03  &               & 6.025E-02    &             & 5.718E-02      &           \\ \hline
        64  &    8.711E-04   &   2.00    &  3.026E-02    &  0.99    &  2.843E-02    &  1.01 \\ \hline
       128 &     2.205E-04  &    1.98   &   1.516E-02   &   1.00   &   1.418E-02  &    1.00      \\ \hline
       256 &     5.514E-05  &    2.00   &   7.588E-03   &   1.00    &  7.079E-03  &    1.00  \\ \hline
                        \multicolumn{7}{c}{By the pressure-robust CR-IFE method,~$p_0=10^6$}  \\ \hline
        32 &    3.478E-03   &              &  6.025E-02    &             & 5.706E-02    &            \\ \hline
        64 &     8.711E-04   &   2.00    &  3.026E-02    &  0.99    &  2.853E-02   &   1.00 \\ \hline
      128 &     2.205E-04   &   1.98    &  1.516E-02    &  1.00    &  1.426E-02   &   1.00     \\ \hline
      256 &     5.514E-05   &   2.00    &  7.588E-03   &   1.00     & 7.132E-03    &  1.00 \\ \hline
     \end{tabular}
}
\end{center}
\end{table}

 \section{Conclusions}\label{sec_con}
In this paper, we have developed and analyzed a pressure-robust IFE method based on the CR element for solving the Stokes interface problem on unfitted meshes. On each interface element, the shape functions for velocity and pressure are jointly modified to satisfy the coupled interface jump conditions, resulting in a coupling between velocity and pressure in the IFE space.
To enhance pressure robustness, we adopt a recently developed approach that replaces the test function on the right-hand side of the discrete formulation with its divergence-conforming reconstruction. Although the velocity and pressure remain coupled in the IFE space, our analysis rigorously proves that this approach still ensures pressure robustness. Furthermore, we establish the stability and optimal error estimates of the proposed method with constants independent of the interface position relative to the mesh.
Numerical experiments were performed to validate the theoretical results. The simulations demonstrate optimal convergence rates in both velocity and pressure, and illustrate pressure robustness.

Future work will focus on extending the proposed method to three-dimensional settings and higher-order elements.
The extension to three dimensions is relatively straightforward from an algorithmic perspective by following the strategy developed in \cite{ji2025mini}.
However, the pressure-robust analysis becomes more challenging due to the increased complexity of the interface geometry and the resulting geometric approximation errors. 
For example, for a face \(F\) of an element, the exact subregions
\(F^\pm:=F\cap\Omega^\pm\) and their discrete counterparts
\(F_h^\pm:=F\cap\Omega_h^\pm\) generally do not coincide, resulting in additional difficulties in the pressure-robust analysis. 
This issue is specific to the three-dimensional case and does not arise in two dimensions.

A higher-order IFE method for Stokes interface problems was proposed in \cite{chen2021p2}. Nevertheless, to the best of our knowledge, the development of pressure-robust IFE methods in the higher-order setting remains an open problem. 
A fundamental difficulty is that a piecewise-linear approximation of the interface is generally insufficient to retain higher-order accuracy. The development of accurate geometric representations and compatible basis functions, together with a pressure-robust formulation in this setting, constitutes an interesting direction for future research.

\noindent\textbf{Funding}~
H. Ji was supported by the NSFC Grant 12371370, Ministry of Education Key Laboratory of
NSLSCS key project 202402, and the NSF of NJUPT Grant NY225138. F. Wang was supported by the NSFC Grants 12371369, 12071227 and the National Key Research and Development Program of China 2020YFA0713803.

\bibliographystyle{plain}

\end{document}